\documentclass{amsart}

\usepackage{latexsym}

\usepackage{amsfonts,amssymb,euscript,epsfig,color}

\usepackage[all]{xy}

\newtheorem{theo}{Theorem}[section]
\newtheorem{defi}[theo]{Definition}

\newtheorem{lem}[theo]{Lemma}
\newtheorem{prop}[theo]{Proposition}
\newtheorem{rem}[theo]{Remark}
\newtheorem{coro}[theo]{Corollary}
\newtheorem{exam}[theo]{Example}

\newcommand{\ggot}{\ensuremath{\mathfrak{g}}}
\newcommand{\hgot}{\ensuremath{\mathfrak{h}}}
\newcommand{\kgot}{\ensuremath{\mathfrak{k}}}

\newcommand{\slgot}{\ensuremath{\mathfrak{sl}}}
\newcommand{\tgot}{\ensuremath{\mathfrak{t}}}
\newcommand{\ugot}{\ensuremath{\mathfrak{u}}}

\newcommand{\Acal}{\ensuremath{\mathcal{A}}}
\newcommand{\Bcal}{\ensuremath{\mathcal{B}}}
\newcommand{\Ccal}{\ensuremath{\mathcal{C}}}

\newcommand{\Lcal}{\ensuremath{\mathcal{L}}}
\newcommand{\Mcal}{\ensuremath{\mathcal{M}}}
\newcommand{\Ncal}{\ensuremath{\mathcal{N}}}
\newcommand{\Ocal}{\ensuremath{\mathcal{O}}}
\newcommand{\Qcal}{\ensuremath{\mathcal{Q}}}

\newcommand{\Xcal}{\ensuremath{\mathcal{X}}}
\newcommand{\Ycal}{\ensuremath{\mathcal{Y}}}

\newcommand{\Ucal}{\ensuremath{\mathcal{U}}}
\newcommand{\Vcal}{\ensuremath{\mathcal{V}}}

\newcommand{\Pbb}{\ensuremath{\mathbb{P}}}
\newcommand{\Z}{\ensuremath{\mathbb{Z}}}
\newcommand{\C}{\ensuremath{\mathbb{C}}}

\newcommand{\R}{\ensuremath{\mathbb{R}}}

\newcommand{\tore}{\ensuremath{\mathbb{T}}}

\newcommand{\fgene}{\ensuremath{\mathcal{C}^{-\infty}}}
\newcommand{\croc}{\ensuremath{\hookrightarrow}}

\newcommand{\indice}{\ensuremath{\hbox{\rm Index}}}

\newcommand{\Cr}{\ensuremath{\hbox{\rm Cr}}}

\newcommand{\Tr}{\ensuremath{\hbox{\rm Tr}}}

\newcommand{\Char}{\ensuremath{\hbox{\rm Char}}}
\newcommand{\End}{\ensuremath{\hbox{\rm End}}}

\newcommand{\Thom}{\ensuremath{\hbox{\rm Thom}}}

\newcommand{\qfor}{\ensuremath{\mathcal{Q}^{-\infty}}}
\newcommand{\Rfor}{\ensuremath{R^{-\infty}}}
\newcommand{\Rforc}{\ensuremath{R^{-\infty}_{tc}}}

\def \K {{\rm \bf K}}
\def \T {{\rm T}}

\def \what {\widehat}
\def \wtde {\widetilde}
\def \indH {{\rm Ind}^{^K}_{_H}}

\def \HolT {{\rm Hol}^{^K}_{_T}}

\def \Hols {{\rm Hol}^{^K}_{_{K_{\sigma}}}}

\def \ad {{\rm ad}}

\def \clif {{\bf c}}

\title{Formal Geometric Quantization II}

\author{Paul-Emile  PARADAN}

\address{Institut de Math\'ematiques et de Mod\'elisation de Montpellier (I3M), 
Universit\'e Montpellier 2} 

\email{Paul-Emile.Paradan@math.univ-montp2.fr}

\date{June 2009}

\begin{document}


\begin{abstract}
In this paper we pursue the study of formal geometric quantization of non-compact Hamiltonian manifolds. Our main result 
is the proof that two quantization process coincide. This fact was obtained by Ma and Zhang in the preprint Arxiv:0812.3989 
by completely different means.
\end{abstract}


\maketitle

{\def\thefootnote{\relax}
\addtocounter{footnote}{-1}
}

{\small
\tableofcontents}

In the previous article \cite{pep-formal}, we have studied some functorial properties of
the ``formal geometric quantization'' process $\qfor$, which is defined on {\em proper Hamiltonian manifolds}, e.g. 
{\em non-compact} Hamiltonian manifolds with {\em proper} moment map. 

There is another way, denoted $\Qcal^\Phi$, of quantizing proper Hamiltonian manifolds by localizing the
index of the Dolbeault Dirac operator on the critical points of the square of the moment map \cite{Ma-Zhang,pep-RR,pep-ENS}.

The main purpose of this paper is to provide a geometric proof 
that the quantization process $\qfor$ and $\Qcal^\Phi$ coincide. This fact was proved by Ma and Zhang in the 
recent preprint \cite{Ma-Zhang} by completely different means.

\medskip


\section{Introduction and statement of results}\label{sec:intro}

Let us first recall the definition of the geometric quantization of a
smooth and compact Hamiltonian manifold. Then we show two way of extending 
the notion of geometric quantization to the case of a \emph{non-compact} 
Hamiltonian manifold. 

\medskip

Let $K$ be a compact connected Lie group, with Lie algebra $\kgot$.
In the Kostant-Souriau framework, a Hamiltonian $K$-manifold
$(M,\Omega,\Phi)$ is pre-quantized if there is an equivariant
Hermitian line bundle $L$ with an invariant Hermitian connection
$\nabla$ such that
\begin{equation}\label{eq:kostant-L}
    \Lcal(X)-\nabla_{X_M}=i\langle\Phi,X\rangle\quad \mathrm{and} \quad
    \nabla^2= -i\Omega,
\end{equation}
for every $X\in\kgot$. Here $X_M$ is the vector field on $M$ defined
by $X_M(m)=\frac{d}{dt} e^{-tX}m|_{0}$.

The data $(L,\nabla)$ is also called a Kostant-Souriau line bundle, and 
$\Phi: M\to\kgot^*$ is the moment map. Remark that 
conditions (\ref{eq:kostant-L}) imply via the equivariant
Bianchi formula the relation
\begin{equation}\label{eq:hamiltonian-action}
    \iota(X_M)\Omega= -d\langle\Phi,X\rangle,\quad X\in\kgot.
\end{equation}

Let us recall the notion of geometric quantization when $M$ is \textbf{compact}.
Choose a $K$-invariant almost complex structure $J$ on $M$ which is
compatible with $\Omega$ in the sense that the symmetric bilinear
form $\Omega(\cdot,J\cdot)$ is a Riemannian metric. Let
$\overline{\partial}_L$ be the Dolbeault operator with coefficients
in $L$, and let $\overline{\partial}_L^*$ be its (formal) adjoint.
The \emph{Dolbeault-Dirac operator} on $M$ with coefficients in $L$
is $D_L= \overline{\partial}_L+\overline{\partial}_L^*$, considered
as an elliptic operator from $\Acal^{0,\textrm{\tiny even}}(M,L)$ to
$\Acal^{0,\textrm{\tiny odd}}(M,L)$. Let $R(K)$ be the representation ring of $K$.

\begin{defi}\label{def:quant-compact-lisse}
 The geometric quantization of a {\em compact} Hamiltonian $K$-manifold $(M,\Omega,\Phi)$ is the element 
 $\Qcal_K(M)\in R(K)$ defined as the equivariant index of the
Dolbeault-Dirac operator $D_L$.
\end{defi}

\medskip 

Let us consider the case of a \textbf{proper} Hamiltonian $K$-manifold $M$:  the manifold is (perhaps) \textbf{non-compact} but 
the moment map $\Phi: M\to \kgot^{*}$ is supposed to be proper. Under this properness assumption, one define the 
{\em formal geometric quantization} of $M$  as an element $\qfor_K(M)$ that belongs to $\Rfor(K)$ \cite{Weitsman,pep-formal}. 
Let us recall the definition.

\medskip

Let $T$ be a maximal torus of $K$. Let $\tgot^*$ be the dual of  the Lie algebra 
of $T$ containing the weight lattice $\wedge^*$ : $\alpha\in \wedge^*$ if $i\alpha:\tgot\to i \R$ is the 
differential of a character of $T$. Let $C_K\subset \tgot^*$ be a Weyl chamber, and let 
$\what{K}:=\wedge^*\cap C_K$ be the set of dominant weights. The ring
of characters $R(K)$ has a $\Z$-basis $V_\mu^K,\mu\in \what{K}$ :
$V_\mu^K$ is the irreducible representation of $K$ with highest weight
$\mu$. 

A representation $E$ of $K$ is  {\em admissible} if it has finite
$K$-multiplicities : \break $\dim(\hom_K(V_\mu^K,E))<\infty$ for
every $\mu\in\what{K}$.  Let $\Rfor(K)$  be the Grothendieck group 
associated to the $K$-admissible representations. We have an inclusion map $R(K)\croc \Rfor(K)$ 
and  $\Rfor(K)$ is canonically identify with $\hom_\Z(R(K),\Z)$. 

For any $\mu\in \what{K}$ which is a regular value of moment map $\Phi$, the
reduced space (or symplectic quotient)
$M_{\mu}:=\Phi^{-1}(K\cdot\mu)/K$ is a {\em compact} orbifold equipped with a
symplectic structure $\Omega_{\mu}$. Moreover
$L_{\mu}:=(L\vert_{\Phi^{-1}(\mu)}\otimes \C_{-\mu})/K_{\mu}$ is a
Kostant-Souriau line orbibundle over $(M_{\mu},\Omega_{\mu})$. The
definition of the index of the Dolbeault-Dirac operator carries over
to the orbifold case, hence $\Qcal(M_{\mu})\in \Z$ is defined. In
Section \ref{subsec:symplectic-quotient}, we explain how this notion of geometric 
quantization extends further to the case of singular symplectic quotients. So the integer 
$\Qcal(M_{\mu})\in \Z$ is well
defined for every $\mu\in \what{K}$: in particular
$\Qcal(M_{\mu})=0$ if $\mu\notin\Phi(M)$.

\begin{defi}\label{def:formal-quant}
Let $(M,\Omega,\Phi)$ be a {\em proper} Hamiltonian $K$-manifold which is 
prequantized by a Kostant-Souriau line bundle $L$. The formal quantization of $(M,\Omega,\Phi)$ is 
the element of $\Rfor(K)$ defined by
$$
\qfor_K(M)=\sum_{\mu\in \what{K}}\Qcal(M_{\mu})\, V_{\mu}^{K}\ .
$$
\end{defi}

\medskip

When $M$ is compact, the fact that 
\begin{equation}\label{eq:Q-R=0}
\Qcal_K(M)=\qfor_K(M)
\end{equation}
is known as the ``quantization commutes with reduction Theorem''. This  was conjectured 
by Guillemin-Sternberg in \cite{Guillemin-Sternberg82} and was first proved by Meinrenken 
\cite{Meinrenken98} and  Meinrenken-Sjamaar \cite{Meinrenken-Sjamaar}. Other proofs 
of (\ref{eq:Q-R=0}) were also given by Tian-Zhang \cite{Tian-Zhang98} and the author 
\cite{pep-RR}. For complete references on the subject the reader should consult 
\cite{Sjamaar96,Vergne-Bourbaki}.

One of the main feature of the formal geometric quantization $\qfor$ is its stability relatively to 
the restriction to subgroups.

\begin{theo}[\cite{pep-formal}]\label{theo:pep-formal}
Let $M$ be a pre-quantized Hamiltonian $K$-manifold which is \emph{proper}. Let $H\subset K$ be
a closed connected Lie subgroup such that $M$ is still \emph{proper} as a
Hamiltonian $H$-manifold. Then $\qfor_K(M)$ is $H$-admissible and we
have $\qfor_K(M)|_H=\qfor_H(M)$ in $\Rfor(H)$.
\end{theo}

\medskip

When $M$ is a proper Hamiltonian $K$-manifold, we can also define another 
``formal geometric quantization'', denoted  
\begin{equation}\label{eq:Q-Phi}
\Qcal^{\Phi}_K(M)\in \Rfor(K),
\end{equation}
by localizing the index of the Dolbeault-Dirac operator $D_L$ on the set $\Cr(\|\Phi\|^2)$ of 
critical points of the square of the moment map (see Section \ref{sec:Qcal-Phi} for the precise definiton).
 We proved in previous papers  
\cite{pep-ENS,pep-formal,pep-hol-series}
that 
\begin{equation}\label{eq:qfor=qfor}
\qfor_K(M)=\Qcal^{\Phi}_K(M).
\end{equation}
in somes situations: 

$\bullet$ $M$ is a coadjoint orbit of a semi-simple Lie group $S$ that 
parametrizes a representation of the discrete series of $S$,

$\bullet$ $M$ is a Hermitian vector space.

\medskip

In her ICM 2006 plenary lecture \cite{VergneICM}, Vergne conjectured that (\ref{eq:qfor=qfor}) holds when 
$\Cr(\|\Phi\|^2)$ is compact. Recently, Ma and Zhang \cite{Ma-Zhang} prove the following generalisation 
of this conjecture.

\medskip

\begin{theo}\label{theo:intro}
The equality (\ref{eq:qfor=qfor}) holds for {\bf any} proper Hamiltonian $K$-manifold.
\end{theo}

\medskip

This article is dedicated to the study of the quantization map $\Qcal^\Phi$:
\begin{enumerate}
\item[$\bullet$] In Section \ref{sec:Qcal-Phi}, we give the precise definition of the quantization process $\Qcal^\Phi$. In particular, we refine the constant $c_\gamma$ appearing in \cite{Ma-Zhang}[Theorem 0.1].
\item[$\bullet$] In Section \ref{subsec:quant-point}, we explain how to compute the quantization of a point.
\item[$\bullet$] We give in Section \ref{sec:preuve} another proof of  Theorem \ref{theo:intro} by using the 
technique of symplectic cutting developped in \cite{pep-formal}. 
\item[$\bullet$] In Section \ref{sec:prop-Q-Phi}, we consider the case where $K=K_1\times K_2$ acts on $M$ in a way that the 
symplectic reduction $M{/\!\!/}_{0}K_1$ is a {\em smooth} proper $K_2$-Hamiltonian manifold. We show then that 
the $K_1$-invariant part of $\Qcal^\Phi_{K_1\times K_2}(M)$ is equal to $\Qcal^{\Phi_2}_{K_2}(M{/\!\!/}_{0}K_1)$.
\end{enumerate}

In Section \ref{sec:quant-K/H}, we study the example where $M$ is the cotangent bundle of a homogeneous space:  
$M=\T^*(K/H)$ where $H$ is a closed subgroup of $K$.  We  see that 
$\T^*(K/H)$ is a proper Hamiltonian $K$-manifold prequantized by the trivial line bundle.  
A direct computation gives 
\begin{equation}\label{eq:Q-Phi-G-G/H}
\Qcal^{\Phi}_K(\T^*(K/H))= {\rm L}^2(K/H) \quad {\rm in}\quad \Rfor(K).
\end{equation}

Let us denoted $[\T^*(K/H)]_{\mu,K}$ the symplectic reduction at $\mu\in\what{K}$ of 
the K-Hamiltonian manifold $\T^*(K/H)$. Theorem \ref{theo:intro} together with (\ref{eq:Q-Phi-G-G/H}) give
$$
\Qcal\left([\T^*(K/H)]_{\mu,K}\right)=\dim [V_\mu^K]^H,
$$
for any $\mu\in\what{K}$. Here $[V_\mu^K]^H\subset V_\mu^K$ is the subspace of $H$-invariant vectors.

Then we consider the action of a closed connected subgroup $G\subset K$ on $\T^*(K/H)$. We first 
check that $\T^*(K/H)$ is a {\em proper} Hamiltonian $G$-manifold  if and only if the restriction 
${\rm L}^2(K/H)\vert_G$ is an admissible $G$-representation. Then, using Theorem \ref{theo:pep-formal},  we get that
\begin{equation}\label{eq:Q-Phi-K-G/H}
\qfor_G(\T^*(K/H))= {\rm L}^2(K/H)\vert_G \quad {\rm in}\quad \Rfor(G).
\end{equation}
In other words, the multiplicity of $V_\lambda^G$ in ${\rm L}^2(K/H)$ is equal to the quantization of the 
reduced space $[\T^*(K/H)]_{\lambda,G}$.

\section{Quantizations of non-compact manifolds}

In this section we define the quantization process $\Qcal^{\Phi}$, and we give another definition 
of the quantization process $\qfor$ that uses the notion of symplectic cutting \cite{pep-formal}.

\subsection{Transversally elliptic symbols}\label{subsec:transversally}

Here we give the basic definitions from the theory of transversally
elliptic symbols (or operators) defined by Atiyah-Singer in
\cite{Atiyah74}. For an axiomatic treatment of the index morphism
see Berline-Vergne \cite{B-V.inventiones.96.1,B-V.inventiones.96.2} and 
Paradan-Vergne \cite{pep-vergne:bismut}. For a short introduction see \cite{pep-RR}.

Let $\Xcal$ be a {\em compact} $K$-manifold. Let $p:\T
\Xcal\to \Xcal$ be the projection, and let $(-,-)_\Xcal$ be a
$K$-invariant Riemannian metric. If $E^{0},E^{1}$ are
$K$-equivariant complex vector bundles over $\Xcal$, a
$K$-equivariant morphism $\sigma \in \Gamma(\T
\Xcal,\hom(p^{*}E^{0},p^{*}E^{1}))$ is called a {\em symbol} on $\Xcal$. The
subset of all $(x,v)\in \T \Xcal$ where\footnote{The map $\sigma(x,v)$ will be also denote 
$\sigma\vert_x(v)$} $\sigma(x,v): E^{0}_{x}\to
E^{1}_{x}$ is not invertible is called the {\em characteristic set}
of $\sigma$, and is denoted by $\Char(\sigma)$.

In the following, the product of a symbol  $\sigma$ 
by a complex vector bundle $F\to M$, is the symbol
$$
\sigma\otimes F
$$
defined by $\sigma\otimes F(x,v)=\sigma(x,v)\otimes {\rm Id}_{F_x}$ from 
$E^{0}_x\otimes F_x$ to $E^{1}_x\otimes F_x$. Note that $\Char(\sigma\otimes F)=\Char(\sigma)$.

Let $\T_{K}\Xcal$ be the following subset of $\T \Xcal$ :
$$
   \T_{K}\Xcal\ = \left\{(x,v)\in \T \Xcal,\ (v,X_{\Xcal}(x))_{_{\Xcal}}=0 \quad {\rm for\ all}\
   X\in\kgot \right\} .
$$

A symbol $\sigma$ is {\em elliptic} if $\sigma$ is invertible
outside a compact subset of $\T \Xcal$ (i.e. $\Char(\sigma)$ is
compact), and is $K$-{\em transversally elliptic} if the
restriction of $\sigma$ to $\T_{K}\Xcal$ is invertible outside a
compact subset of $\T_{K}\Xcal$ (i.e. $\Char(\sigma)\cap
\T_{K_2}\Xcal$ is compact). An elliptic symbol $\sigma$ defines an
element in the equivariant $\K$-theory of $\T\Xcal$ with compact
support, which is denoted by $\K_{K}(\T \Xcal)$, and the
index of $\sigma$ is a virtual finite dimensional representation of
$K$, that we denote $\indice^K_{\Xcal}(\sigma)\in R(K)$
\cite{Atiyah-Segal68,Atiyah-Singer-1,Atiyah-Singer-2,Atiyah-Singer-3}.

Let  
$$
\Rforc(K)\subset R^{-\infty}(K)
$$ 
be the $R(K)$-submodule formed by all the infinite sum $\sum_{\mu\in\what{K}}m_\mu V_\mu^K$ 
where the map $\mu\in\what{K}\mapsto m_\mu\in\Z$ has at most a {\em
polynomial} growth. The $R(K)$-module $\Rforc(K)$ is the Grothendieck group 
associated to the {\em trace class} virtual $K$-representations: we can associate 
to any $V\in\Rforc(K)$, its trace $k\to \Tr(k,V)$ which is a generalized function on $K$ 
invariant by conjugation. Then the trace defines a morphism of $R(K)$-module
\begin{equation}\label{eq:trace}
\Rforc(K)\croc \fgene(K)^K.
\end{equation}

A $K$-{\em transversally elliptic} symbol $\sigma$ defines an
element of $\K_{K}(\T_{K}\Xcal)$, and the index of
$\sigma$ is defined as a trace class virtual representation of $K$, that we still denote 
$\indice^K_\Xcal(\sigma)\in \Rforc(K)$. 

Remark that any elliptic symbol of $\T \Xcal$ is $K$-transversally
elliptic, hence we have a restriction map $\K_{K}(\T
\Xcal)\to \K_{K}(\T_{K}\Xcal)$, and a commutative
diagram
\begin{equation}\label{indice.generalise}
\xymatrix{ \K_{K}(\T \Xcal) \ar[r]\ar[d]_{\indice^K_\Xcal}
&
\K_{K}(\T_{K}\Xcal)\ar[d]^{\indice^K_\Xcal}\\
R(K)\ar[r] & \Rforc(K)\ .
   }
\end{equation}

\medskip

Using the {\em excision property}, one can easily show that the
index map $\indice^K_\Ucal: \K_{K}(\T_{K}\Ucal)\to
\Rforc(K)$ is still defined when $\Ucal$ is a
$K$-invariant relatively compact open subset of a
$K$-manifold (see \cite{pep-RR}[section 3.1]).

\medskip

Suppose now that the group $K$ is equal to the product $K_1\times K_2$. When a symbol $\sigma$ 
is $K_1\times K_2$-transversaly elliptic we will be interested in the $K_1$-invariant part of its index, 
that we denote
$$
\left[\indice^{K_1\times K_2}_\Xcal(\sigma)\right]^{K_1}\in \Rforc(K_2).
$$

An intermediate notion between the ``ellipticity'' and ``$K_1\times K_2$-transversal ellipticity'' is 
the ``$K_1$-transversal ellipticity''. When a $K_1\times K_2$-equivariant morphism $\sigma$ is 
{\em $K_1$-transversally elliptic}, its index $\indice^{K_1\times K_2}_\Xcal(\sigma)\in \Rforc(K_1\times K_2)$, 
viewed as a generalized function on $K_1\times K_2$, is {\em smooth} relatively to the variable in $K_2$.
It implies that $\indice^{K_1\times K_2}_\Xcal(\sigma)=\sum_{\lambda} \theta(\lambda)\otimes V_\lambda^{K_1}$ 
with 
$$
\theta(\lambda)\in R(K_2),\quad \forall\ \lambda\in\what{K_1}.
$$
In particular, we know that 
$$
\left[\indice^{K_1\times K_2}_\Xcal(\sigma)\right]^{K_1}=\theta(0)
$$
belongs to $R(K_2)$.

\medskip


Let us recall the multiplicative property of the index map for the product
of manifolds that was proved by Atiyah-Singer in \cite{Atiyah74}.  Consider a compact Lie group $K_2$ acting on two
manifolds $\Xcal_1$ and $\Xcal_2$, and assume that another compact
Lie group $K_1$
acts on $\Xcal_1$ commuting with the action of $K_2$.

The external product of complexes on $\T\Xcal_1$ and $\T\Xcal_2$ induces
a multiplication (see \cite{Atiyah74,pep-vergne:bismut}):
$$
\odot: K_{K_1\times K_2}(\T_{K_1} \Xcal_1)\times K_{K_2}(\T_{K_2} \Xcal_2)
\longrightarrow K_{K_1\times K_2}(\T_{K_1\times K_2} (\Xcal_1\times \Xcal_2)).
$$

The following property will be used frequently in the paper.

\begin{theo}[Multiplicative property] \label{theo:multiplicative-property}
For any $[\sigma_1]\in K_{K_1\times K_2}(\T_{K_1} \Xcal_1)$ and
any $[\sigma_2]\in K_{K_2}(\T_{K_2} \Xcal_2)$ we have
$$
\indice^{K_1\times K_2}_{\Xcal_1\times \Xcal_2}([\sigma_1]\odot[\sigma_2])
=\indice^{K_1\times K_2}_{\Xcal_1}([\sigma_1])\otimes\indice^{K_2}_{\Xcal_2}([\sigma_2]).
$$
\end{theo}

\medskip


We will use in this article the notion of {\em support} of a generalized character.

\begin{defi}
The {\em support} of $\chi:=\sum_{\mu\what{K}}a_\mu V_\mu^K\in\Rfor(K)$ is the set of 
$\mu\in\what{K}$ such that $a_\mu\neq 0$.
\end{defi}

We will say that $\chi\in\Rfor(K)$ is supported outside $B\subset\tgot^*$ if the support of $\chi$
does not intersect $B$. Note that an infinite sum $\sum_{i\in I}\chi_i$ converges in $\Rfor(K)$ if for each ball 
$$
B_r=\{\xi\in\tgot^*\ \vert\ \|\xi\|< r\}
$$ 
the set $\left\{i\in I \ \vert\ \mathrm{support}(\chi_i)\cap B_r\neq \emptyset\right\}$ 
is finite. 

\begin{defi}\label{def:O(r)}
We denote by $O(r)$ any character of $\Rfor(K)$ which is supported outside the ball $B_r$. 
\end{defi}

\subsection{Definition and first properties of $\Qcal^{\Phi}$}\label{sec:Qcal-Phi}

Let $(M,\Omega,\Phi)$ be a proper Hamiltonian $K$-manifold prequantized by an 
equivariant line bundle $L$. Let $J$ be an invariant almost complex structure 
compatible with $\Omega$. Let $p:\T M\to M$ be the projection.

Let us first describe the principal symbol  of the Dolbeault-Dirac operator $\overline{\partial}_L+\overline{\partial}_L^*$.  
The complex vector bundle $(\T^* M)^{0,1}$ is $K$-equivariantly identified 
with the tangent bundle $\T M$ equipped with the complex structure $J$.
Let $h$ be the  Hermitian structure on  $(\T M,J)$ defined by : 
$h(v,w)=\Omega(v,J w) - i \Omega(v,w)$ for $v,w\in \T M$. The symbol 
$$
\Thom(M,J)\in 
\Gamma\left(M,\hom(p^{*}(\wedge_{\C}^{even} \T M),\,p^{*}
(\wedge_{\C}^{odd} \T M))\right)
$$  
at $(m,v)\in \T M$ is equal to the Clifford map
\begin{equation}\label{eq.thom.complex}
 \clif_{m}(v)\ :\ \wedge_{\C}^{even} \T_m M
\longrightarrow \wedge_{\C}^{odd} \T_m M,
\end{equation}
where $\clif_{m}(v).w= v\wedge w - \iota(v)w$ for $w\in 
\wedge_{\C}^{\bullet} \T_{m}M$. Here $\iota(v):\wedge_{\C}^{\bullet} 
\T_{m}M\to\wedge^{\bullet -1} \T_{m}M$ denotes the 
contraction map relative to $h$. Since $\clif_{m}(v)^2=-\| v\|^2 {\rm Id}$, the map  
$\clif_{m}(v)$ is invertible for all $v\neq 0$. Hence the characteristic set 
of $\Thom(M,J)$ corresponds to the $0$-section of $\T M$.

It is a classical fact that the principal symbol  of the Dolbeault-Dirac operator 
$\overline{\partial}_{L} + \overline{\partial}^*_{L}$ is equal to\footnote{Here 
we use an identification $\T^*M\simeq \T M$ given by an invariant Riemannian metric.}  
\begin{equation}\label{eq:tau-Lambda}
\Thom(M,J)\otimes L,
\end{equation}
see \cite{Duistermaat96}.  Here also we have $\Char(\Thom(M,J)\otimes L)=0-{\rm section\ of}\ \T M$.

\begin{rem}
When the manifold $M$ is a product $M_1\times M_2$ the symbol $\Thom(M,J)\otimes L$ is equal to the product 
$\sigma_1\odot\sigma_2$ where $\sigma_k=\Thom(M_k,J_k)\otimes L_k$. 
\end{rem}

When $M$ is compact, the symbol $\Thom(M,J)\otimes L$ is elliptic and
then defines an element of the equivariant \textbf{K}-group of
$\T M$. The topological index of $\Thom(M,J)\otimes L\in\K_{K}(\T M)$ is equal to the analytical index of the
Dolbeault-Dirac operator $\overline{\partial}_L+\overline{\partial}_L^*$ :
\begin{equation}\label{eq:index-analytique-topologique}
    \Qcal_{K}(M)=\indice^K_M(\Thom(M,J)\otimes L)\quad
    \mathrm{in}\quad R(K).
\end{equation}

\medskip

When $M$ is not compact the topological index of $\Thom(M,J)\otimes L$
is not defined. In order to extend the notion of geometric quantization to this 
setting we deform the symbol $\Thom(M,J)\otimes L$ in the ``Witten'' way \cite{pep-RR,pep-ENS}. Consider the identification
$\xi\mapsto\wtde{\xi},\kgot^*\to\kgot$ defined by a $K$-invariant scalar product
on  $\kgot^*$. We define the {\em Kirwan vector field} on
$M$ : 
\begin{equation}\label{eq-kappa}
    \kappa_m= \left(\wtde{\Phi(m)}\right)_M(m), \quad m\in M.
\end{equation}

\begin{defi}\label{def:pushed-sigma}
The symbol  $\Thom(M,J)\otimes L$ pushed by the vector field $\kappa$ is the symbol $\clif^\kappa$ 
defined by the relation
$$
\clif^\kappa\vert_m(v)=\Thom(M,J)\otimes L\vert_m(v-\kappa_m)
$$
for any $(m,v)\in\T M$.
\end{defi}

Note that $\clif^\kappa\vert_m(v)$ is invertible except if
$v=\kappa_m$. If furthermore $v$ belongs to the subset $\T_K M$
of tangent vectors orthogonal to the $K$-orbits, then $v=0$ and
$\kappa_m=0$.  Indeed $\kappa_m$ is tangent to $K\cdot m$ while
$v$ is orthogonal.

Since $\kappa$ is the Hamiltonian vector field of the function
$\frac{-1}{2}\|\Phi\|^2$, the set of zeros of $\kappa$ coincides with the set
$\Cr(\|\Phi\|^2)$ of critical points of $\|\Phi\|^2$. Finally we have 
$$
\Char(\clif^\kappa)\cap \T_K M \simeq
\Cr(\|\Phi\|^2).
$$

In general $\Cr(\|\Phi\|^2)$ is not compact, so $\clif^\kappa$ does not define a transversally elliptic 
symbol on $M$. In order to define a kind of index of $\clif^\kappa$, we proceed as follows. 
For any invariant open relatively compact subset 
$U\subset M$  the set  $\Char(\clif^\kappa\vert_{U})\cap \T_K U \simeq \Cr(\|\Phi\|^2)\cap U$ is compact when  
\begin{equation}\label{eq:condition-U}
\partial \Ucal\cap \Cr(\|\Phi\|^2)=\emptyset.
\end{equation}
When (\ref{eq:condition-U}) holds we denote 
\begin{equation}\label{eq:Q-Phi-U}
\Qcal^{\Phi}_K(U):= \indice^K_{U}(\clif^\kappa\vert_{U})\quad
   \in\quad \Rforc(K)
\end{equation}
the equivariant index of the transversally elliptic symbol $\clif^\kappa\vert_{U}$.

It will be usefull to understand the dependance of the generalized character 
$\Qcal^{\Phi}_K(U)$ relatively to the data $(U,\Omega,L)$. So we consider two proper Hamilonian 
$K$-manifolds $(M,\Omega,\Phi)$ and $(M',\Omega',\Phi')$ respectively prequantized by the 
line bundles $L$ and $L'$. Let $V\subset M$ and $V'\subset M'$ two invariant open  subsets.

\begin{prop}\label{prop:psi-invariance}
$\bullet$ The generalized character $\Qcal^{\Phi}_K(U)$ does not depend of the choice of an invariant almost 
complex structure on $U$ which is compatible with $\Omega\vert_U$.

$\bullet$ Suppose that there exists an equivariant diffeomorphism $\Psi: V\to V'$ such that 
\begin{enumerate}
\item $\Psi^*(\Phi')=\Phi$,
\item $\Psi^*(L')=L$,
\item there exists an homotopy of symplectic forms taking $\Psi^*(\Omega'\vert_{V'})$  to $\Omega\vert_V$.
\end{enumerate}

Let $U'\subset \overline{U'}\subset V'$ be an invariant open relatively compact subset such that $\partial U'$ 
satisfies (\ref{eq:condition-U}). 
Take $U=\Psi^{-1}(U')$. Then $\partial U$ satisfies (\ref{eq:condition-U}) and 
$$
\Qcal^{\Phi'}_K(U')=\Qcal^{\Phi}_K(U) \quad \in \quad \Rfor(K).
$$
\end{prop}

\begin{proof} Let us prove the first point. Let $\clif^{\kappa}_i\vert_{U}, i=0,1$ be the transversally elliptic symbols 
defined with the compatible almost complex structure $J_i, i=0,1$. Since the space of compatible almost complex structure 
is contractible, there exist an homotopy $J_t, t\in[0,1]$ of almost complex structures linking $J_0$ and $J_1$. If we use 
Lemma 2.2 in \cite{pep-RR}, we know that there exists an invertible bundle map $A\in\Gamma(U,\End(\T U))$, homotopic 
to the identity,  such that $A\circ J_0= J_1\circ A$. With the help of $A$ we prove then that 
the symbols $\clif_0^\kappa\vert_{U}$ and $\clif_1^\kappa\vert_{U}$ define the same class in 
$\K_K(\T_K U)$ (see \cite{pep-RR}[Lemma 2.2]). Hence their equivariant index coincide.

Let us prove the second point.The characters $\Qcal^{\Phi}_K(U)$ and $\Qcal^{\Phi'}_K(U')$ are computed as the 
equivariant index of the symbols $\clif^\kappa\vert_{U}$ and $\clif^{\kappa'}\vert_{U'}$. 
Let $\tilde{\clif}^\kappa\vert_{U}$ the pull back of $\clif^{\kappa'}\vert_{U'}$ by $\Psi$. Thanks to 
the point $(1)$ and $(2)$, the only thing which differs in the definitions of the 
symbols $\clif^\kappa\vert_{U}$ and $\tilde{\clif}^\kappa\vert_{U}$ are the almost complex structures $J$ and 
$\tilde J=\Psi^*(J')$ : the first one is comptible with $\Omega$ and the second one with 
$\Psi^*(\Omega'\vert_{V'})$. Since these two symplectic structure are homotopic, one sees that 
the almost complex structures $J$ and $\tilde J$ are also homotopic. So we can conclude like in the first point. 
\end{proof}

Let us recall the basic fact concerning the singular values of $\|\Phi\|^2$.
\begin{lem}
The set of singular values of $\|\Phi\|^2: M\to\R$ forms a sequence $0\leq r_1< r_2<\ldots < r_k<\ldots$  
which is finite iff $\Cr(\|\Phi\|^2)$ is compact. In the other case $\lim_{k\to\infty}r_k=\infty$.
\end{lem}

At each regular value $R$ of $\Cr(\|\Phi\|^2)$, we associate the invariant open subset 
$M_{< R}:=\{\|\Phi\|^2<R\}$ which satisfies (\ref{eq:condition-U}). The restriction $\clif^\kappa\vert_{M_{< R}}$ defines then 
a transversally elliptic symbol on $M_{< R}$: let $\Qcal^{\Phi}_K(M_{< R})$ be its equivariant index.

 Let us show that $\Qcal^{\Phi}_K(M_{< R})$ has a limit when $R\to\infty$.
The set $\Cr(\|\Phi\|^2)$ has the following decomposition
\begin{equation}\label{eq:crit-Phi}
\Cr(\|\Phi\|^2)=\bigcup_{\beta\in\Bcal} \ \underbrace{K\cdot (M^{\wtde{\beta}}\cap \Phi^{-1}(\beta))}_{Z_\beta}
\end{equation}
where the $\Bcal$ is a subset of the Weyl chamber $\tgot^*_+$. Note that each part $Z_\beta$ is compact, 
hence $\Bcal$ is  finite only if $\Cr(\|\Phi\|^2)$ is compact. When $\Cr(\|\Phi\|^2)$ is non-compact, 
the set $\Bcal$ is infinite, but it is easy to see that $\Bcal\cap B_r$  
is finite for any $r\geq 0$. For any $\beta\in \Bcal$, we consider a relatively compact open 
invariant neighborhood 
$\Ucal_\beta$ of  $Z_\beta$ such that 
$\Cr(\|\Phi\|^2)\cap \overline{\Ucal_\beta}= Z_\beta$. 

\begin{defi}\label{def:Q-beta}
We denote 
$$
\Qcal^{\beta}_K(M)\in\Rforc(K)
$$
the index\footnote{The index of $\clif^\kappa\vert_{\Ucal_\beta}$ was denoted $RR^{^K}_\beta(M,L)$ in \cite{pep-RR}.} 
of the transversally elliptic symbol $\clif^\kappa\vert_{\Ucal_\beta}$.
\end{defi}

A simple application of the excision property \cite{pep-RR} gives that 
\begin{equation}\label{eq:Q-Phi-R}
\Qcal^{\Phi}_K(M_{< R}):\
= \sum_{\|\beta\|^2< R} \Qcal^{\beta}_K(M).
\end{equation}

\medskip

We have now the key fact 

\begin{theo}\label{prop:Q-beta-support}
The generalized character $\Qcal^{\beta}_K(M)$ is supported outside the 
open ball $B_{\|\beta\|}$. 
\end{theo}

\medskip

\begin{proof}
Proposition \ref{prop:Q-beta-support} follows directly from the computations done in \cite{pep-RR}. 
First consider the case where $\beta\neq 0$ is a $K$-invariant element of $\Bcal$. 
Let $i: \tore_\beta \croc T$ be the compact torus generated by $\beta$. If $F$ is $\Z$-module 
we denote $F\widehat{\otimes}\, \Rfor(\tore_{\beta})$ the $\Z$-module formed by 
the infinite formal sums $\sum_{a} E_{a}\, h^a$ taken over the set of weights of $\tore_{\beta}$, 
where $E_{a}\in F$ for every $a$.

Since $\tore_\beta$ lies in the center of $K$, the morphism $\pi:(k,t)\in K\times \tore_\beta \mapsto 
kt\in K$ induces a map $\pi^*:\Rfor(K)\to \Rfor(K)\,\widehat{\otimes}\, \Rfor(\tore_{\beta})$.

The normal bundle $\Ncal$ of $M^{\wtde{\beta}}$ in 
$M$ inherits a canonical complex structure $J_{\Ncal}$ on the fibers.  We denote by 
$\overline{\Ncal}\to M^{\wtde{\beta}}$ the complex vector bundle with the 
opposite complex structure.  The torus $\tore_{\beta}$ is included in 
the center of $K$, so the bundle $\overline{\Ncal}$ and the virtual 
bundle $\wedge_{\C}^{\bullet}\overline{\Ncal}:=\wedge_{\C}^{even} 
\overline{\Ncal}\stackrel{0}{\to}\wedge_{\C}^{odd}\overline{\Ncal}$ 
carry a $K\times \tore_{\beta}$-action: they can be considered as 
elements of $K_{K\times \tore_{\beta}}(M^{\wtde{\beta}})\, = 
\,K_{K}(M^{\wtde{\beta}})\otimes R(\tore_{\beta})$.  


In \cite{pep-RR}, we have defined an inverse of 
$\wedge_{\C}^{\bullet}\overline{\Ncal}$,
$\left[ \wedge_{\C}^{\bullet} \overline{\Ncal}\,\right]^{-1}_{\beta}\in 
K_{K}(M^{\wtde{\beta}})\,\widehat{\otimes}\, \Rfor(\tore_{\beta})$, which is 
polarized by $\beta$.  It means that $\left[ \wedge_{\C}^{\bullet} 
\overline{\Ncal}\,\right]^{-1}_{\beta}= \sum_{a} N_{a}\, h^a$ with 
$N_{a}\neq 0$ only if $\langle a, \beta\rangle\geq 0$.  

We prove in \cite{pep-RR} the following localization formula :
\begin{equation}\label{eq.localisation.2}
\pi^*\left[\Qcal^{\beta}_K(M)\right]=RR^{^{K\times \tore_{\beta}}}_{\beta}
\left(M^{\wtde{\beta}},L_{\vert M^{\beta}}\otimes 
\left[ \wedge_{\C}^{\bullet} \overline{\Ncal}\,\right]^{-1}_{\beta}\right)\ ,
\end{equation}
as an equality in $R^{-\infty}(K)\,\widehat{\otimes}\, \Rfor(\tore_{\beta})$. With 
(\ref{eq.localisation.2}) in hand, it is easy to see that $V_\mu^K$ occurs 
in the character $\Qcal^{\beta}_K(M)$ only  if $(\mu,\beta)\geq \|\beta\|^2$ (See Lemma 9.4 in 
\cite{pep-RR}). 

\medskip

Now we consider the case were $\beta\in\Bcal$ is not a $K$-invariant element. Let $\sigma$ be the 
unique open face of the Weyl chamber $\tgot^*_+$ which contains $\beta$. Let $K_\sigma$ be the 
corresponding stabilizer subgroup. We consider the symplectic slice $\Ycal_\sigma\subset M$: 
it is a $K_\sigma$ invariant Hamiltonian submanifold of $M$ which is prequantized by the line bundle 
$L\vert_{\Ycal_\sigma}$.  The restriction of $\Phi$ to $\Ycal_\sigma$ is a moment map 
$\Phi_\sigma:\Ycal_\sigma\to\kgot_\sigma^*$ which is proper in a neighborhood of 
$\beta\in\kgot_\sigma^*$. The set 
$$
K_\sigma\cdot(\Ycal_\sigma^{\wtde{\beta}}\cap\Phi_\sigma^{-1}(\beta))=M^{\wtde{\beta}}\cap\Phi^{-1}(\beta)
$$
is a component of $\Cr(\|\Phi_\sigma\|^2)$. Let $\Qcal^{\beta}_{K_\sigma}(\Ycal_\sigma)\in\Rforc(K_\sigma)$
be the corresponding character (see Definition \ref{def:Q-beta}). 

We prove in \cite{pep-RR}[Section 7], the following induction formula: 
\begin{equation}\label{eq:induction-slice}
\Qcal^{\beta}_{K}(M)=\Hols\left(\Qcal^{\beta}_{K_\sigma}(\Ycal_\sigma)\right)
\end{equation}
where $\Hols: \Rfor(K_\sigma)\to\Rfor(K)$ is the holomorphic induction map. See the Appendix in \cite{pep-RR} 
for the definition and properties of these induction maps. 

We know from the previous case that 
$$
\Qcal^{\beta}_{K_\sigma}(\Ycal_\sigma)=\sum_{\mu\in\what{K_\sigma}} m_\mu V^{K_\sigma}_\mu
$$
where $m_\mu\neq 0\Longrightarrow (\mu,\beta)\geq \|\beta\|^2$. Then, with 
(\ref{eq:induction-slice}), we get 
\begin{eqnarray*}
\Qcal^{\beta}_{K_\sigma}(\Ycal_\sigma)&=&\sum_{(\mu,\beta)\geq \|\beta\|^2}
m_\mu\ \Hols(V^{K_\sigma}_\mu)\\
&=&\sum_{(\mu,\beta)\geq \|\beta\|^2}m_\mu \ \HolT(t^\mu),
\end{eqnarray*}
where $\HolT: \Rfor(T)\to\Rfor(K)$ is the holomorphic induction map.
 
Let $\rho$ be half the sum of the positive roots. The term $\HolT(t^\mu)$ is equal to $0$ when 
$\mu+\rho$ is not a regular element of $\tgot^*$. When $\mu+\rho$ is a regular element of $\tgot^*$, 
we have $\HolT(t^\mu)=(-1)^{|\omega|} V^K_{\mu_\omega}$ where 
$$
\mu_\omega=\omega(\mu+\rho)-\rho
$$
is dominant for a unique $\omega\in W$. 

Finally, a representation $V^K_{\lambda}$ appears in the character $\Qcal^{\beta}_{K}(M)$ only 
if $\lambda= \mu_\omega$ for a weight $\mu$ satisfying  $(\mu,\beta)\geq \|\beta\|^2$. 
Hence, for such $\lambda$, we have
\begin{eqnarray*}
 \|\lambda\|&=&  \|\mu+\rho-\omega^{-1}\rho\|\\
 &\geq&(\mu+\rho-\omega^{-1}\rho,\frac{\beta}{\|\beta\|})\\
 &\geq& \|\beta\|.
 \end{eqnarray*}
In the last inequality we use that $(\rho-\omega^{-1}\rho,\beta)\geq 0$ 
since $\rho-\omega^{-1}\rho$ is a sum of positive roots, and $\beta\in\tgot^*_+$.

\end{proof}

\medskip

 With the help of Theorem \ref{prop:Q-beta-support} and decomposition (\ref{eq:Q-Phi-R}), 
we see that the multiplicity of $V_\gamma^K$ in $\Qcal^{\Phi}_K(M_{< R})$ does not depend on 
the regular value $R>\|\gamma\|^2$. We can refine the constant $c_\gamma$ appearing in 
\cite{Ma-Zhang}[Theorem 0.1]: take $c_\gamma$ equal to $\|\gamma\|^2$ instead of\footnote{Here $\rho$ is half the sum 
of the positive roots. Hence  $\|\gamma+\rho\|^2-\|\rho\|^2-\|\gamma\|^2=2(\rho,\gamma)\geq 0$ and $(\rho,\gamma)=0$ 
only if the weight $\gamma$ belongs to the center of $\kgot\simeq\kgot^*$.}
$\|\gamma+\rho\|^2-\|\rho\|^2\geq \|\gamma\|^2$.

\begin{defi}\label{def:Q-Phi}
The generalized character $\Qcal^{\Phi}_K(M)$ is defined as the limit in $\Rfor(K)$ of 
$\Qcal^{\Phi}_K(M_{< R})$  when $R$ goes to infinity. In other 
words
\begin{equation}\label{eq:Q-Q-beta}
\Qcal^{\Phi}_K(M)= \sum_{\beta\in\Bcal} \Qcal^{\beta}_K(M).
\end{equation}
\end{defi}

Note that for any regular value $R$ of $\|\Phi\|^2$ we have the useful relation 
\begin{equation}\label{eq-Qcal-Phi-R}
\Qcal^{\Phi}_K(M)=\Qcal^{\Phi}_K(M_{<R}) +O(\sqrt{R}).
\end{equation}

\subsection{Quantization of a symplectic quotient}\label{subsec:symplectic-quotient}

We will now explain how we define the geometric quantization of
\emph{singular} compact Hamiltonian manifolds : here ``singular"
means that the manifold is obtained by symplectic reduction.

Let $(N,\Omega)$ be a smooth symplectic manifold equipped with a
Hamiltonian action of $K_1\times K_2$ : we denote $(\Phi_1,\Phi_2): N\to
\kgot_1^*\times\kgot_2^*$ the corresponding moment map. We assume that
$N$ is pre-quantized by a $K_1\times K_2$-equivariant line bundle $L$
and we suppose that the map $\Phi_1$ is \textbf{proper}. One wants
to define the geometric quantization of the (compact) symplectic quotient
$$
N{/\!\!/}_{0}K_1:=\Phi_1^{-1}(0)/K_1.
$$

Let $\kappa_1$ be the Kirwan vector field attached to the moment map $\Phi_1$.  We denote  by 
$\clif^{\kappa_1}$ the symbol $\Thom(N,J)\otimes L$ pushed by the vector field $\kappa_1$. For any regular value 
$R_1$ of $\|\Phi_1\|^2$, we consider the restriction $\clif^{\kappa_1}\vert_{N_{<R_1}}$ to the invariant, 
 open subset $N_{< R_1}:=\{\|\Phi_1\|^2< R_1\}$. The symbol $\clif^{\kappa_1}\vert_{N_{< R_1}}$ is 
$K_1\times K_2$-equivariant and $K_1$-transversally elliptic, hence we can consider its index 
$$
\indice^{K_1\times K_2}_{N_{< R_1}}(\clif^{\kappa_1}\vert_{N_{< R_1}})\in \Rfor(K_1\times K_2).
$$
which is smooth relatively to the parameter in $K_2$.  We consider the following extension of Definition 
\ref{def:Q-Phi}.

\begin{defi}\label{def:quant-phi-1}
The generalized character $\Qcal^{\Phi_1}_{K_1\times K_2}(N)$ is defined as the limit in 
$\Rfor(K_1\times K_2)$ of 
$\indice^{K_1\times K_2}_{N_{< R_1}}(\clif^{\kappa_1}\vert_{N_{< R_1}})$  when $R_1$ goes to infinity. 
\end{defi}

Here $\Cr(\|\Phi_1\|^2)$ is equal to the disjoint union of the compact $K_1\times K_2$-invariant 
subsets  $Z_{\beta_1}:= K_1\cdot (M^{\wtde{\beta_1}}\cap \Phi_1^{-1}(\beta_1))$, $\beta_1\in\Bcal_1$.  
For $\beta_1\in \Bcal_1$, we consider an invariant relatively compact open subset $\Ucal_{\beta_1}$ such that:
$Z_{\beta_1}\subset \Ucal_{\beta_1}$ and $Z_{\beta_1}= \Cr(\|\Phi_1\|^2)\cap\overline{\Ucal_{\beta_1}}$. Let 
$\Qcal^{\beta_1}_{K_1\times K_2}(N)\in \Rfor(K_1\times K_2)$ be the equivariant index of the $K_1$-transversally 
elliptic symbol $\clif^L_{\kappa^1}\vert_{\Ucal_{\beta_1}}$.  The $K_1$-transversallity condition imposes that 
$\Qcal^{\beta_1}_{K_1\times K_2}(N)=\sum_{\lambda}\theta^{\beta_1}(\lambda)\otimes V_{\lambda}^{K_1}$
with 
$$
\theta^{\beta_1}(\lambda)\in R(K_2),\quad \forall\lambda\in\what{K_1}.
$$


We have the following extension of Theorem \ref{prop:Q-beta-support} 

\begin{theo}\label{prop:Q-beta-support-1}
We have $\Qcal^{\beta_1}_{K_1\times K_2}(N)=\sum_{\lambda\in\what{K_1}}\theta^{\beta_1}(\lambda)
\otimes V_{\lambda}^{K_1}$ where $\theta^{\beta_1}(\lambda)\neq 0$ only if 
$\|\lambda\|\geq \|\beta_1\|$.
\end{theo}
\begin{proof}The proof works exactly like the one of Theorem \ref{prop:Q-beta-support}.
\end{proof}


Let us explain the ``quantization commutes with reduction theorem", or why we 
can consider the geometric quantization of 
$$
N{/\!\!/}_{0}K_1:=\Phi_1^{-1}(0)/K_1
$$ 
as the $K_1$-invariant part of $\Qcal^{\Phi_1}_{K_1\times K_2}(N)$.

Let us first suppose that $0$ is a regular value of $\Phi_1$. Then  $N{/\!\!/}_{0}K_1$ is a
compact symplectic \emph{orbifold} equipped with a Hamiltonian
action of $K_2$ : the corresponding moment map is induced by the
restriction of $\Phi_2$ to $\Phi_1^{-1}(0)$. The symplectic quotient
$N{/\!\!/}_{0}K_1$ is pre-quantized by the line orbibundle
$$
L_0:= \left(L|_{\Phi_1^{-1}(0)}\right)/K_1.
$$
Definition \ref{def:quant-compact-lisse} extends to the orbifold
case. We can still define the geometric quantization of $N{/\!\!/}_{0}K_1$  
as the index of an elliptic operator : we denote it by 
$\Qcal_{K_2}(N{/\!\!/}_{0}K_1)\in R(K_2)$. We have 

\begin{theo}\label{theo:Q-R-lisse}
If $0$ is a regular value of $\Phi_1$, the $K_1$-invariant part of 
$\Qcal^{\Phi_1}_{K_1\times K_2}(N)$ is equal to $\Qcal_{K_2}(N{/\!\!/}_{0}K_1)\in R(K_2)$.
\end{theo}

\medskip

Suppose now that $0$ is not a regular value of $\Phi_1$. Let $T_1$
be a maximal torus of $K_1$, and let $C_1\subset \tgot_1^*$ be a Weyl
chamber. Since $\Phi_1$ is proper, the convexity Theorem says that
the image of $\Phi_1$ intersects $C_1$ in a closed locally polyhedral 
convex set, that we denote $\Delta_{K_1}(N)$ \cite{L-M-T-W}.

We consider an element $a\in \Delta_{K_1}(N)$ which is generic and
sufficiently close to $0\in \Delta_{K_1}(N)$ : we denote $(K_1)_a$ the
subgroup of $K_1$ which stabilizes $a$. When $a\in\Delta_{K_1}(N)$  is
generic, one can show (see \cite{Meinrenken-Sjamaar}) that
$$
N{/\!\!/}_{a}K_1:=\Phi_{K_1}^{-1}(a)/(K_1)_a
$$
is a compact Hamiltonian $K_2$-orbifold, and that
$$
L_a:= \left(L|_{\Phi_{K_1}^{-1}(a)}\right)/(K_1)_a.
$$
is a $K_2$-equivariant line orbibundle over $N{/\!\!/}_{a}K_1$ : we can
then define, like in Definition \ref{def:quant-compact-lisse}, the
element $\Qcal_{K_2}(N{/\!\!/}_{a}K_1)\in R(K_2)$ as the equivariant index
of the Dolbeault-Dirac operator on $N{/\!\!/}_{a}K_1$ (with coefficients in $L_a$).

\begin{theo}\label{theo:Q-sing}
The $K_1$-invariant part of $\Qcal^{\Phi_1}_{K_1\times K_2}(M)$ is equal to $\Qcal_{K_2}(N{/\!\!/}_{a}K_1)
\in R(K_2)$. In particular,  the elements $\Qcal_{K_2}(N{/\!\!/}_{a}K_1)$ do not depend on the choice
of the generic element $a\in\Delta_H(N)$, when $a$ is sufficiently
close to $0$.
\end{theo}

\textsc{Proofs of Theorem \ref{theo:Q-R-lisse} an Theorem \ref{theo:Q-sing} }. When $N$ is compact and 
$K_2=\{e\}$, the proofs can be found in \cite{Meinrenken-Sjamaar} and in \cite{pep-RR}. Let us explain 
briefly how the $\K$-theoretic proof of \cite{pep-RR} extends naturally to our case. Like in Definition 
\ref{def:Q-Phi}, we have the following decomposition
$$
\Qcal^{\Phi_1}_{K_1\times K_2}(N)= \sum_{\beta\in\Bcal_1} \Qcal^{\beta_1}_{K_1\times K_2}(N),
$$
And Theorem \ref{prop:Q-beta-support-1} tells us that 
$\left[\Qcal^{\beta_1}_{K_1\times K_2}(N)\right]^{K_1}=0$ if $\beta_1\neq 0$. We have proved the first step: 
$$
\left[\Qcal^{\Phi_1}_{K_1\times K_2}(N)\right]^{K_1}=\left[\Qcal^{0}_{K_1\times K_2}(N)\right]^{K_1}.
$$

The analysis of the term $\left[\Qcal^{0}_{K_1\times K_2}(N)\right]^{K_1}$ is undertaken in \cite{pep-RR} 
when $K_2=\{e\}$: we explain that this term is equal either to $\Qcal(N{/\!\!/}_{0}K_1)$ when $0$ is a 
regular value, or to $\Qcal(N{/\!\!/}_{a}K_1)$ with $a$ generic. It work similarly with an action of 
a compact Lie group $K_2$. $\Box$

\begin{defi}\label{def:quant-sing}
The geometric quantization of $N{/\!\!/}_{0}K_1:=\Phi_1^{-1}(0)/K_1$ is taken as the 
$K_1$-invariant part of $\Qcal^{\Phi_1}_{K_1\times K_2}(N)$. We denote it $\Qcal_{K_2}(N{/\!\!/}_{0}K_1)$.
\end{defi}

\medskip

\subsection{Quantization of points }\label{subsec:quant-point}

Let $(M,\Omega,\Phi)$ be a proper Hamiltonian $K$-manifold prequantized by a Kostant-Souriau line bundle $L$. 
Let $\mu\in \what{K}$ be dominant weight such that $\Phi^{-1}(K\cdot\mu)$ is a $K$-orbit in $M$. Let 
$m^o\in \Phi^{-1}(\mu)$ so that 
$$
\Phi^{-1}(K\cdot\mu)=K\cdot m^o
$$
Then the reduced space $M_\mu:=\Phi^{-1}(K\cdot\mu)/K$ 
is a point. The aim of this section is to compute the quantization of $M_\mu$: $\Qcal(M_\mu)\in\Z$.

Let $H$ be the stabilizer subgroup of $m^o$.  We have a linear action of $H$ on the $1$-dimensional 
vector space $L_{m^o}\subset L$. We have $H\subset K_\mu$ where $K_\mu$ is the connected subgroup of $K$ 
that fixes $\mu\in\tgot^*$. Let $\C_{-\mu}$ be the $1$-dimensional representation of $K_\mu$ associated to 
the infinitesimal character $-i\mu$.

Let us denote $\chi$ be the character of $H$ defined by the $1$-dimensional representation $\C_\chi:=L_{m^o}\otimes \C_{-\mu}$. 
We know from the Kostant formula (\ref{eq:kostant-L}) that $\chi=1$ on the identity component $H^o\subset H$.

\begin{theo}\label{theo:Q-point}
We have 
\begin{equation}\label{eq:quant-point}
\Qcal(M_\mu)=
\begin{cases}
   1\quad {\rm if}\ \chi=1\ {\rm on}\ H\\
   0\quad {\rm in\ the\ other\ case}.
\end{cases}
\end{equation}
\end{theo}

This Theorem tells us in particular that $\Qcal(M_\mu)=1$ when the stabiliser subgroup $H\subset K$ of  a point 
$m^o\in\Phi^{-1}(\mu)$ is {\em connected}.

\begin{proof} Let $N=M\times \overline{K\cdot\mu}$ be the proper Hamiltonian $K$-manifold which is prequantized
by the line bundle $L_N:= L\otimes \left[\C_{-\mu}\right]$. Let us denote $\Phi_N$ the moment map on $N$. Since 
$\Phi^{-1}(K\cdot\mu)$ is a $K$-orbit in $M$, we see that $\Phi_N^{-1}(0)$ is the $K$-orbit through $n^o:=(m^o,\mu)$ where 
$m^o\in\Phi^{-1}(\mu)$. Note that $H$ is the stabilizer subgroups of $n^o$.

Let $\Qcal^{\Phi_N}_K(N)\in\Rfor(K)$ be the formal quantization of $N$ through the proper map $\Phi_N$. By definition 
\begin{eqnarray*}
\Qcal(M_\mu)&=& \left[\Qcal^{\Phi_N}_K(N)\right]^K\\
&=& \left[\Qcal^{0}_K(N)\right]^K.
\end{eqnarray*}
where $\Qcal^{0}_K(N)$ depends only of a neighborhood of $\Phi_N^{-1}(0)$.

The orbit $K\cdot n^o\croc N$ is an isotropic embedding since it is the $0$-level of the moment map $\Phi_N$. Then to 
describe a $K$-invariant neighborhood of $K\cdot n^o$ in $N$ we 
can use the normal-form recipe of Marle, Guillemin and Sternberg.

First we consider, following Weinstein (see 
\cite{Guillemin-Sternberg90,Weinstein83}), 
the symplectic normal bundle
\begin{equation}\label{symp-normal}
\mathcal{V}:=\T(K\cdot n^o)^{\perp,\Omega}\Big/ 
\T(K\cdot n^o)\ ,
\end{equation}
where the orthogonal (\,${}^{\perp,\Omega}$) is taken relatively to the symplectic 2-form. 
We have 
$$
\mathcal{V}=K\times_{H}V
$$
where the vector space $V:=\T_{n^o}(K\cdot n^o)^{\perp,\Omega}\Big/ \T_{n^o}(K\cdot n^o)$ 
inherits a symplectic structure and an Hamiltonian action of the group $H$: we denote 
$\Phi_H: V\to\hgot^*$ the corresponding moment map.

\medskip

Consider now the following symplectic manifold 
\begin{equation}\label{Y-tilde-sigma}
\tilde{N}:= \mathcal{V}\oplus 
\T^*(K/H)= K\times_H
\Big( (\kgot/\hgot)^*\oplus V \Big).
\end{equation}
The action of $H$ on
$\tilde{N}$ is Hamiltonian and the moment map 
$\Phi_{\tilde{N}}:\tilde{N}\to\kgot^*$ is given by 
the equation
\begin{equation}
\Phi_{\tilde{N}}([k;\xi,v])=k\cdot(\xi+\Phi_H(v))\quad k\in K,\ 
\xi\in (\kgot/\hgot)^*,\ v\in V\ .
\end{equation}
The Hamiltonian $K$-manifold $\tilde{N}$ is prequantized by the line bundle $L_{\tilde{N}}:= K\times_H\C_\chi$.
\medskip

The {\it local normal form} Theorem (see \cite{Guillemin-Sternberg84}, \cite{Sjamaar-Lerman91} Proposition
2.5 ) tells us that there exists a $K$-Hamiltonian isomorphism $\Upsilon :\mathcal{U}_1 \stackrel{\sim}{\to} \mathcal{U}_2$
between a $K$-invariant neighborhood  $\mathcal{U}_1$ of $K\cdot n^o$ in $N$, and
a $K$-invariant neighborhood $\mathcal{U}_2$ of $K/H$ in $\tilde{N}$. 
This isomorphism $\Upsilon$, when restricted to $K\cdot n^o$, 
corresponds to the natural isomorphism $K\cdot n^o\stackrel{\sim}{\to}K/H$. 

Thanks to $\Upsilon$, we know that the fiber $\Phi_H^{-1}(0)\subset V$ is reduced to $\{0\}$. This last 
point is equivalent to the fact that $\Phi_H$ (and then $\Phi_{\tilde{N}}$) is proper map (see \cite{pep-formal}). 
We check easily that the set of critical points of $\|\Phi_{\tilde{N}}\|^2$ is reduced to $\Phi_{\tilde{N}}^{-1}(0)=K/H$. 
Then, thank to the isomorphism $\Upsilon$, we have that 
\begin{equation}\label{eq:Qcal-N-tilde}
\Qcal^{0}_K(N)=\Qcal^{0}_K(\tilde{N})=\Qcal^{\Phi_{\tilde{N}}}_K(\tilde{N}).
\end{equation}

Let $\indH: \Rfor(H)\to\Rfor(K)$ be the induction map that is defined by the relation 
$\langle\indH(\phi), E\rangle=\langle\phi, E\vert_H\rangle$ 
for any $\phi\in \Rfor(H)$ and $E\in R(K)$. Note that 
$$
[\indH(\phi)]^K=\langle\indH(\phi), \C\rangle=\langle\phi, \C\rangle=[\phi]^H.
$$

Since $\Phi_H:V\to\hgot^*$ is proper one can consider the quantization 
of the vector space $V$ through the map $\Phi_H$: $\Qcal^{\Phi_H}_H(V)\in\Rfor(H)$. 

\begin{prop}\label{prop:3-points}
$\bullet$ We have 
\begin{equation}\label{eq:indH}
\Qcal^{\Phi_{\tilde{N}}}_K(\tilde{N})=\indH\left(\Qcal^{\Phi_H}_H(V)\otimes \C_\chi\right)
\end{equation}

$\bullet$ The formal quantization $\Qcal^{\Phi_H}_H(V)$ coincides, as a generalized $H$-module,  
to the $H$-module $S(V^*)$ of polynomial function on $V$.

$\bullet$ The set $\left[S(V^*)\right]^{H^o}$ of polynomials invariant by the connected component $H^o$ is reduced 
to the scalars.	
\end{prop}

With the last Proposition we can finish the proof of Theorem \ref{theo:Q-point} as follows. We have 
\begin{eqnarray*}
\Qcal(M_\mu)&=&\left[\Qcal^{\Phi}_K(N)\right]^K\\
&=&\left[\Qcal^{\Phi_{\tilde{N}}}_K(\tilde{N})\right]^K\\
&=& \left[\Qcal^{\Phi_H}_H(V)\otimes \C_\chi\right]^H\\
&=& \left[S(V^*)\otimes \C_\chi\right]^H=\left[\C_\chi\right]^H.
\end{eqnarray*}

\begin{proof}
The first point of Proposition \ref{prop:3-points} follows from the property of induction defined by Atiyah (see Section 3.4 in \cite{pep-RR}). 
Let us explain the arguments. We work with the $H$-manifold $\Ycal=(\kgot/\hgot)^*\oplus V$ and the $H$-equivariant map 
$j:\Ycal \croc \tilde{N}:= K\times_H \Ycal, y\mapsto [e,y]$. 

We notice\footnote{\label{eq.espace.tangent} These identities come from the following $K\times H$-equivariant 
isomorphism of vector bundles over $K\times\Ycal$:
$\T_{H}(\tilde{N})\to K\times(\kgot/\hgot \oplus \T \Ycal),
(k,m;\frac{d}{dt}_{\vert t=0}(ke^{tX})+ v_{m})\mapsto
(k,m; pr_{\kgot/\hgot}(X)+v_{m})$. Here $pr_{\kgot/\hgot}:
\kgot\to \kgot/\hgot$ is the orthogonal projection.}
that $\T\tilde{N}\simeq 
K\times_{H}(\kgot/\hgot \oplus \T \Ycal)$, and  
that $\T_{K}\tilde{N}\simeq K\times_{H}( \T_{H} \Ycal)$. Hence the 
map $j$ induces an isomorphism $j_{*}: K_{H}(\T_{H}\Ycal)\to K_{K}(\T_{K} \tilde{N})$. 
Theorem 4.1 of Atiyah \cite{Atiyah74} tells us that the following diagram
\begin{equation}\label{eq:diagramme-induit}
\xymatrix{
K_{H}(\T_{H}\Ycal)\ar[r]^{j_{*}}\ar[d]_{\indice_{\Ycal}^H} & 
K_{K}(\T_{K}\tilde{N})\ar[d]^{\indice_{\tilde{N}}^K}\\
\Rfor(H)\ar[r]_{\indH} &  \Rfor(K)\ .
   }
\end{equation} 
is commutative. 

The tangent bundle 
$\T\tilde{N}$ is equivariantly diffeomorphic to 
$$
K\times_H \left[\kgot/\hgot \oplus (\kgot/\hgot)^*\oplus \T V\right]\simeq 
K\times_H \left[(\kgot/\hgot)_\C \oplus \T V\right]
$$
where $(\kgot/\hgot)_\C$ is the complexification of the real vector space $\kgot/\hgot$.
We consider on $\tilde{N}$ the almost complex structure $J_{\tilde{N}}= (i, J_V)$ where 
$i$ is the complex structure on  $(\kgot/\hgot)_\C$ and $J_V$ is a compatible (constant) 
complex structure on the symplectic vector space $V$. Note that $J_{\tilde{N}}$ is compatible 
with the symplectic structure on a neighborhood $U$ of the $0$-section of the bundle
$\tilde{N}\to K/H$.

Let $\kappa_{\tilde{N}}$ be the Kirwan vector field on $\tilde{N}$: 
$$
\kappa_{\tilde{N}}([k;\xi,v])= -\xi+ i \,[\xi,\Phi_H(v)]\oplus \kappa_V(v)\quad \in\quad 
(\kgot/\hgot)_\C \oplus V.
$$
Here $\kappa_V$ is the Kirwan vector field relative to the Hamiltonian action of $H$ on the symplectic 
vector space $V$. Note that $\kappa_{\tilde{N}}$ vanishes exactly on the $0$-section of the bundle
$\tilde{N}\to K/H$.

Let $\clif^{\kappa_{\tilde{N}}}$ be the symbol $\Thom(\tilde{N},J_{\tilde{N}})\otimes L_{\tilde{N}}$ 
pushed by the vector field $\kappa_{\tilde{N}}$. The generalized character $\Qcal^{\Phi_{\tilde{N}}}_K(\tilde{N})$ 
is either computed as the equivariant index of the symbols $\clif^{\kappa_{\tilde{N}}}$ or $\clif^{\kappa_{\tilde{N}}}\vert_U$. 

\begin{rem} The fact that $J_{\tilde{N}}$ is not compatible on the entire manifold $\tilde{N}$ is not problematic, since 
$J_{\tilde{N}}$ is  compatible in a neighborhood U of the set where $\kappa_{\tilde{N}}$ vanishes. See the first point of 
Lemma \ref{prop:psi-invariance}.
\end{rem}

For $X+i\eta \oplus w\in \T_{[k;\xi,v]}\tilde{N}\simeq (\kgot/\hgot)_\C \oplus V$, the map  
\begin{equation}\label{eq:clif-tilde-N}
\clif^{\kappa_{\tilde{N}}}(X+i\eta \oplus w)=\clif\Big(X+\xi+i(\eta-[\xi,\Phi_H(v)]\Big)\odot\clif\Big(w-\kappa_V(v)\Big)
\end{equation}
acts on the vector space $\wedge_\C (\kgot/\hgot)_\C \otimes \wedge_{J_V} V\otimes\C_\chi$. 

Let ${\rm Bott}(\kgot/\hgot)$ be the Bott morphism of the vector space $\kgot/\hgot$. It is an elliptic morphism defined by
$$
{\rm Bott}(\kgot/\hgot)\vert_{\xi}(\eta)=\clif(\xi+i\eta)\quad {\rm acting\ on}\quad \wedge_\C (\kgot/\hgot)_\C,
$$
for $\eta\in \T_\xi(\kgot/\hgot)$. Let $\clif^{\kappa_{V}}$ be the symbol $\Thom(V,J_{V})$ pushed by
 the vector field $\kappa_{V}$.
\begin{lem}
We have 
$$
\clif^{\kappa_{\tilde{N}}}=j_*\Big({\rm Bott}(\kgot/\hgot)\odot\clif^{\kappa_{V}}\otimes \C_\chi\Big).
$$
\end{lem}

\begin{proof} We work with the symbol 
$$
\sigma^T\vert_{(\xi,v)}(\eta)=\clif(\xi+i\eta- iT\,[\xi,\Phi_H(v)])
$$
acting on $\wedge_\C (\kgot/\hgot)_\C$. Note that 
${\rm Bott}(\kgot/\hgot)=\sigma^0$. From (\ref{eq:clif-tilde-N}), we see that 
$\clif^{\kappa_{\tilde{N}}}=j_*\Big(\sigma^1\odot\clif^{\kappa_{V}}\otimes \C_\chi\Big)$. It is now easy to 
check that $\sigma^T\odot\clif^{\kappa_{V}}\otimes \C_\chi, T\in [0,1]$ is an homotopy of transversally elliptic 
symbols on  $\kgot/\hgot\times V$.
\end{proof}

The commutative diagram (\ref{eq:diagramme-induit}) and the last Lemma gives 
\begin{eqnarray*}
\Qcal^{\Phi_{\tilde{N}}}_K(\tilde{N})&=&
\indice_{\tilde{N}}^K(\clif^{\kappa_{\tilde{N}}})\\
&=&\indH\left(\indice_{\kgot/\hgot\times V}^K\Big({\rm Bott}(\kgot/\hgot)\odot\clif^{\kappa_{V}}\Big)\otimes \C_\chi\right)\\
&=&\indH\left(\indice_{\kgot/\hgot}^K({\rm Bott}(\kgot/\hgot))
\otimes\indice_{V}^K(\clif^{\kappa_{V}})\otimes \C_\chi\right)\\
&=& \indH\left(\Qcal^{\Phi_{H}}_H(V)\otimes \C_\chi\right).
\end{eqnarray*}
We have used here that the equivariant index of ${\rm Bott}(\kgot/\hgot)$ is equal to $1$ (e.g. the trivial representation).

Let us proved now the second point of Proposition \ref{prop:3-points}. The Kirwan vector field $\kappa^V$ 
satisfies the simple rule:
\begin{equation}\label{eq:kappa-J}
(\kappa^V(v), J_V v)=-\Omega(\kappa^V(v),v)= \frac{1}{2}\|\Phi_H(v)\|^2, \quad v\in V.
\end{equation}
It shows in particular that $\kappa^V(v)=0\Leftrightarrow \Phi_H(v)=0$. Since the moment map 
$\Phi_H:V\to\hgot^*$ is quadratic, the
fact that $\Phi_H$ is proper is equivalent to the fact that $\Phi_H^{-1}(0)=0$. 

We consider on $V$ the family of symbol $\sigma^s$ :
$$
\sigma^s\vert_v(w)=\clif\Big(w-s\kappa^V(v)-(1-s)J_V v\Big) 
$$
viewed as a map from $\wedge^{\rm even}_\C V$ to $\wedge^{\rm odd}_\C V$. Thanks to (\ref{eq:kappa-J}), one sees that $\sigma^s$
is a family of $K$-transversally elliptic symbol on $V$. Hence $\sigma^1=\clif^{\kappa_{V}}$ and $\sigma^0=\clif(w-J_V v) $ defines the same 
class in the group $K_K(\T_K V)$. The symbol $\sigma^0$ was first studied by Atiyah \cite{Atiyah74} when $\dim_\C V=1$. The author 
considered the general case in \cite{pep-RR}. We have 
$$
\indice^K_V(\sigma^0)= S(V^*)\quad \hbox{\rm in}Ê\ \Rfor(K).
$$

The last point of Proposition \ref{prop:3-points} is a consequence of the properness of the moment map $\Phi_H$ (see Section 5 of 
\cite{pep-formal}).

\end{proof}
\end{proof}

\begin{exam}[\cite{pep-formal}] We consider the action of the unitary group ${\rm U}_n$ on $\C^n$. The symplectic 
form on $\C^n$ is defined by $\Omega(v,w)=\frac{i}{2}\sum_k v_k\overline{w_k}-\overline{v_k}w_k$. Let us identify
the Lie algebra $\ugot_n$ with its dual through the trace map. The moment map
$\Phi: \C^n\to\ugot_n$ is defined by $\Phi(v)=\frac{1}{2i}v\otimes v^*$ where $v\otimes v^*:\C^n\to\C^n$ is the linear 
map $w\mapsto (\sum_k \overline{v_k}w_k) v$. One checks easily that the pull-back by $\Phi$ of a ${\rm U}_n$-orbit in 
$\ugot_n$ is either empty or a ${\rm U}_n$-orbit in $\C^n$. We knows also that the stabiliser subgroup of a non-zero 
vector of $\C^n$ is connected since it is diffeomorphic to ${\rm U}_{n-1}$. Finally we have 
\begin{equation}\label{eq:quant-point-example}
\Qcal((\C^n)_\mu)=
\begin{cases}
   1\quad {\rm if}\ \mu\in \what{{\rm U}_n}\ {\rm belongs \ to\ the\ image\ of}\ \Phi\\
   0\quad {\rm if}\ \mu\in \what{{\rm U}_n}\ {\rm does\ not\ belongs\ to\ the\ image\ of}\ \Phi.
\end{cases}
\end{equation}

Then one checks that $\qfor_{{\rm U}_n}(\C^n)$ coincides in $\Rfor({\rm U}_n)$ with the algebra $S((\C^n)^*)$ of polynomial function on $\C^n$.
\end{exam}

\begin{exam}[\cite{pep-hol-series}] We consider the Lie group ${\rm SL}_2(\R)$ and its compact torus of dimension $1$ 
denoted by $T$. The Lie algebra $\slgot_2(\R)$ is identified with its dual through the trace map, and the Lie algebra 
$\tgot$ is naturally identified with $\slgot_2(\R)^T$. 
For $l\in \Z\setminus \{0\}$, we consider the character $\chi_l$ of $T$ defined by 
$$
\chi_l\left(\begin{array}{cc}\cos\theta& -\sin\theta\\ \sin\theta& \cos\theta\end{array}\right)= e^{il\theta}.
$$
Its differential $\frac{1}{i}d\chi_l\in\tgot^*$ correspond (through the trace map) to the matrix 
$$
X_l=\left(\begin{array}{cc}0 & l/2\\ -l/2& 0\end{array}\right).
$$

Let $\Ocal_l$ be the coadjoint orbit of the group ${\rm SL}_2(\R)$ trough the matrix $X_l$. It is a Hamiltonian 
${\rm SL}_2(\R)$-manifold prequantized by the  ${\rm SL}_2(\R)$-equivariant line bundle $L_l\simeq {\rm SL}_2(\R)\times_T \C_{l}$, 
where $\C_l$ is the $T$-module 
associated to the character $\chi_l$. We look at the Hamiltonian action 
of $T$ on $\Ocal_l$. Let $\Phi_T:\Ocal_l\to \tgot^*$ be the corresponding moment map. One checks that 
the moment map $\Phi_T$ is {\em proper} and that its image is equal to the half-line $\{ a X_l, a\geq 1\}\subset\tgot^*$. 

We check  that for each $\xi\in \{ a X_l, a\geq 1\}$ the fiber $\Phi^{-1}_T(\xi)$ is equal to a $T$-orbit in $\Ocal_l$. For $k\in \Z$, let us 
denote $\left(\Ocal_l\right)_k$ the symplectic reduction of $\Ocal_l$ at the level $X_k$. We knows that 
$\left(\Ocal_l\right)_k=\emptyset$ if $k\notin\{ a l, a\geq 1\}$, and that $\left(\Ocal_l\right)_k$ is a point if $k\in\{ a l, a\geq 1\}$.

In order to compute $\Qcal(\left(\Ocal_l\right)_k)$ we look at the stabilizer subgroup $T_m:=\{t\in T\,\vert\, t\cdot m=m\}$ for each 
point $m\in\Ocal_l$. One sees that $T_m=T$ if $m=X_l$ and $T_m$ is equal to the center $\{\pm Id\}$ of ${\rm SL}_2(\R)$, 
when $m\neq X_l$.

Theorem \ref{theo:Q-point} gives in this setting that, for $k\in\{ a l, a\geq 1\}$, 
\begin{equation}\label{eq:quant-point-example-sl-2}
\Qcal(\left(\Ocal_l\right)_k)=
\begin{cases}
   1\quad {\rm if}\ l-k \ {\rm is\ even}\\
   0\quad {\rm if}\ l-k \ {\rm is\ odd}.
\end{cases}
\end{equation}

Hence the formal geometric quantization of the proper $T$-manifold $\Ocal_l$ is
\begin{equation}\label{eq:quant-example-sl-2}
\qfor_T(\Ocal_l)=
\begin{cases}
   \C_l\cdot\sum_{p\geq 0} \C_{2p}\quad {\rm if}\ l>0\\
   \C_l\cdot\sum_{p\geq 0} \C_{-2p}\quad {\rm if}\ l<0.
\end{cases}
\end{equation}
Here we recognizes that $\qfor_T(\Ocal_l)$ coincides with the restriction of the holomorphic (resp. anti-holomorphic) discrete series representation  
$\Theta_l$ to the group $T$ when $l>0$ (resp. $l<0$).

\end{exam}

\subsection{Wonderful compactifications and symplectic cuts }\label{subsec:wonderfull}

Another equivalent definition of the quantization $\qfor$ uses a generalisation of the technique of symplectic
cutting (originally due to Lerman \cite{Lerman-cut}) that was introduced in \cite{pep-formal} and was 
motivated by the wonderful compactifications of De Concini and Procesi. Let us recall the method. 

We recall that $T$ is a maximal torus in the compact connected Lie group $K$, and $W$ is the Weyl group. 
We define a {\em $K$-adapted} polytope in $\tgot^*$ to be a $W$-invariant Delzant polytope $P$ in $\tgot^*$
whose vertices are regular elements of the weight lattice $\Lambda^*$. If $\{\lambda_1,\ldots,\lambda_N\}$
are the dominant weights lying in the union of all the closed one-dimensional faces of
$P$, then there is a $G \times G$-equivariant embedding of $G= K_\C$ into
$$\Pbb( \bigoplus_{i=1}^N V^*_{\lambda_i} \otimes V_{\lambda_i})$$
associating to $g \in G$ its representation on $\bigoplus_{i=1}^N V_{\lambda_i}$. 
The closure $\Xcal_{P}$ of the image of $G$ in this
projective space is smooth and is equipped with a $K\times K$ that we denote:
$$
(k_1,k_2)\cdot x= k_2\cdot x\cdot  k_1^{-1}.
$$
Let $\Omega_{\Xcal_P}$ be the symplectic $2$-form on $\Xcal_P$ which given by the Kahler structure. 
We recall briefly the different properties of $(\Xcal_P,\Omega_{\Xcal_P})$ : all the details can be found in \cite{pep-formal}.

\begin{enumerate}
\item $\Xcal_P$ is equipped with an Hamiltonian action of $K\times K$. Let 
$\Phi=(\Phi_l,\Phi_r): M\to\kgot^*\times\kgot^*$ be the corresponding moment map.

\item The image of $\Phi:=(\Phi_l,\Phi_r)$ is equal to $\{(k\cdot\xi,-k'\cdot\xi)\ \vert\ \xi\in P\ 
{\rm and}\ k,k'\in K\}$.

\item The Hamiltonian manifold $(\Xcal_P, K\times K)$ has no multiplicities: the pull-back by $\Phi$ of a 
$K\times K$-orbit in the image is a $K\times K$-orbit in $\Xcal_P$.
\end{enumerate} 

Let $\Ucal_P:= K\cdot P^\circ$ where $P^\circ$ is the interior of $P$. We define
$$
\Xcal_P^\circ:=\Phi_l^{-1}(\Ucal_P)
$$
which is an invariant, open and dense subset of $\Xcal_P$. We have the following 
important property concerning $\Xcal_P^\circ$.

\begin{enumerate}
\item[(4)] There exists an equivariant diffeomorphism $\Upsilon  : K\times \Ucal_P\to \Xcal_P^\circ$
such that $\Upsilon^*(\Phi_l)(k,\xi)= k\cdot \xi$ and $\Upsilon^*(\Phi_r)(k,\xi)=-\xi$.

\item[(5)] This diffeomorphism $\Upsilon$ is a 
quasi-symplectomorphism in the sense that there is a homotopy of
symplectic forms taking the symplectic form on the open subset $K\times \Ucal_P$ of the cotangent bundle $\T^*K$
to the pullback of the symplectic form $\Omega_{\Xcal_P}$ on $\Xcal_P^o$. 

\item[(6)] The symplectic manifold $ (\Xcal_P,\Omega_{\Xcal_P})$ is prequantized by the restriction 
of the hyperplane line bundle $\Ocal(1)\to\Pbb( \oplus_{i=1}^N V^*_{\lambda_i} \otimes V_{\lambda_i})$ to $\Xcal_P$: 
let us denoted $L_P$ the corresponding 
$K\times K$-equivariant line bundle.

\item[(7)] The pull-back of the line bundle $L_P$ by the map $\Upsilon  : K\times \Ucal_P\croc \Xcal_P$ 
is trivial.
\end{enumerate}

\bigskip

Let $(M,\Omega_M,\Phi_M)$ be a proper Hamiltonian $K$-manifold. We
 also consider the Hamiltonian $K\times K$-manifold $\Xcal_P$
associated to a $K$-adapted polytope $P$. We consider now the
product $M\times \Xcal_P$ with the following $K\times K$ action:
\begin{itemize}
  \item the action $k\cdot_1 (m,x)=(k\cdot m,x\cdot k^{-1})$ : the
  corresponding moment map is $\Phi_1(m,x)=\Phi_M(m)+\Phi_r(x)$,
  \item the action  $k\cdot_2 (m,x)=(m,k\cdot x)$ : the
  corresponding moment map is $\Phi_2(m,x)=\Phi_l(x)$.
\end{itemize}

\begin{defi}\label{def:M-P}
    We denote $M_P$ the symplectic reduction at $0$ of $M\times \Xcal_P$
    for the action $\cdot_1$ : $M_P:=(\Phi_1)^{-1}(0)/(K,\cdot_1)$.
\end{defi}

Then $M_{P}$ inherits a Hamiltonian $K$-action with moment map
$\Phi_{M_P}: M_P\to \kgot^*$ whose image is
$\Phi(M) \cap K\cdot P$.

One checks that $M_P$ contains an open and dense subset of smooth points which quasi-symplectomorphic to 
the open subset $(\Phi_M)^{-1}(\Ucal_P)$. If the polytope $P$ is fixed, we can work with the dilated polytopes 
$nP$ for $n\geq 1$. We have then the family of compact, perhaps singular, $K$-hamiltonian manifolds $M_{nP}$, 
$n\geq 1$: in Section \ref{subsec:symplectic-quotient}, we have explained how was defined their geometric quantization 
$\Qcal_K(M_{nP})\in R(K)$. 

We have a convenient definition for $\qfor$.

\begin{prop}[\cite{pep-formal}]\label{prop:defi-2-qfor}
We have the following equality in $\Rfor(K)$:
\begin{equation}\label{eq:defi-2-qfor}
\qfor_K(M)=\lim_{n\to\infty}\Qcal_K(M_{nP}).
\end{equation}
\end{prop}

\medskip

\section{Proof of Theorem \ref{theo:intro} }\label{sec:preuve}

The main result of this section is 

\begin{theo}\label{theo:fondamental}
Let $r_P:=\inf_{\xi\in\partial P}\|\xi\|$. The generalized character  
$$
\Qcal^\Phi_K(M)-\Qcal_K(M_P)\quad \in\Rfor(K)
$$
is supported outside the ball $B_{r_P}$.
\end{theo}

Then, for the dilated polytope $nP, n\geq 1$, the character $\Qcal^\Phi_K(M)-\Qcal^K(M_{nP})$ is supported outside the ball $B_{nr_P}$. 
Taking the limit when $n$ goes to infinity gives 
\begin{equation}\label{eq:Q-Phi-lim}
\Qcal^\Phi_K(M)=\lim_{n\to\infty}\Qcal_K(M_{nP}).
\end{equation}
Finally, the identity of Theorem \ref{theo:intro},
$$
\Qcal^\Phi_K(M)=\qfor_K(M), 
$$
is a direct consequence of (\ref{eq:defi-2-qfor}) and (\ref{eq:Q-Phi-lim}).

\bigskip

Recall that $O(r)\in \Rfor(K)$ denoted any generalized character supported outside the ball $B_r$. 

\bigskip

Theorem \ref{theo:fondamental}  follows from the comparison of three differents geometrical situation. All of them concern 
Hamiltonian actions of $K_1\times K_2$, where $K_1$ and $K_2$ are two copies of $K$.

\medskip

\medskip 

{\bf First setting.} We work with the Hamiltonian $K_1\times K_2$-manifold $M\times \Xcal_P$: here $K_1$ acts both on $M$ and on 
$\Xcal_P$. Since the moment map $\Phi_1$ (relative to the $K_1$-action)  
is proper we may``quantize''  $M\times \Xcal_P$ via the map  $\|\Phi_1\|^2$ : let  
$$
\Qcal^{\Phi_1}_{K_1\times K_2}(M\times \Xcal_P)\in \Rfor(K_1\times K_2)
$$
be the corresponding generalized character. Recall that $\Qcal_{K_2}(M_P)$ is equal to \break 
$[\Qcal^{\Phi_1}_{K_1\times K_2}(M\times \Xcal_P)]^{K_1}$.

\medskip

\medskip 

{\bf Second setting.}
We consider the same setting than before : the Hamiltonian action of $K_1\times K_2$ on $M\times \Xcal_P$. But 
we ``quantize''  $M\times \Xcal_P$ through the global moment map $\Phi=(\Phi_1,\Phi_2)$. Here we have some liberty 
in the choice of the scalar product on $\kgot^*_1\times\kgot^*_2$. If $\|\xi\|^2$ is an invariant Euclidean norm on $\kgot^*$, 
we take on $\kgot^*_1\times\kgot^*_2$ the Euclidean norm 
\begin{equation}\label{eq:norm-rho}
\|(\xi_1,\xi_2)\|^2_\rho=\|\xi_1\|^2+\rho\|\xi_2\|^2
\end{equation}
depending on a parameter $\rho>0$. Let us consider the quantization of $M\times \Xcal_P$ via the map $\|\Phi\|^2_\rho$: 
$$
\Qcal^{\Phi,\rho}_{K_1\times K_2}(M\times \Xcal_P)\in \Rfor(K_1\times K_2).
$$

\medskip 

\medskip

{\bf Third setting.} We consider the cotangent bundle $\T^*K$ with the Hamiltonian action of $K_1\times K_2$: $K_1$ acts 
by {\em right} translations, and $K_2$ by {\em left} translations. We consider the Hamiltonian action of 
$K_1\times K_2$ on $M\times\T^* K$ : here $K_1$ acts both on $M$ and on $\T^*K$. Let $\Phi=(\Phi_1,\Phi_2)$ be the 
global moment map on $M\times\T^* K$. Since the moment map $\Phi$ is 
proper we can ``quantize''  $M\times \T^*K$ via the map  $\|\Phi\|_\rho^2$ : let  
$$
\Qcal^{\Phi,\rho}_{K_1\times K_2}(M\times \T^*K)\in \Rfor(K_1\times K_2)
$$
be the corresponding generalized character.

\medskip

\medskip

Theorem \ref{theo:fondamental} is a consequence of the following propositions.

\medskip

First we compare $\Qcal^{\Phi}_{K_2}(M)$ with the $K_1$-invariant part of $\Qcal^{\Phi,\rho}_{K_1\times K_2}(M\times \T^*K)$.

\begin{prop}\label{prop:etape1}
For any $\rho \in ]0,1]$, we have 
\begin{equation}\label{eq:etape1}
\left[\Qcal^{\Phi,\rho}_{K_1\times K_2}(M\times \T^*K)\right]^{K_1}=\Qcal^{\Phi}_{K_2}(M) \quad {\rm in}\quad \Rfor(K_2).
\end{equation}
\end{prop}

\medskip

Then we compare the $K_1$-invariant part of the generalized characters \break  
$\Qcal^{\Phi,\rho}_{K_1\times K_2}(M\times \T^*K)$ and $\Qcal^{\Phi,\rho}_{K_1\times K_2}(M\times \Xcal_P)$.

\begin{prop}\label{prop:etape2}
For any $\rho \in ]0,1]$, we have the following relation in $\Rfor(K_2)$ 
\begin{equation}\label{eq:etape2}
\left[\Qcal^{\Phi,\rho}_{K_1\times K_2}(M\times \Xcal_P) \right]^{K_1} -
\left[\Qcal^{\Phi,\rho}_{K_1\times K_2}(M\times \T^*K)\right]^{K_1}
= O(r_P)
\end{equation}
\end{prop}

\medskip

Finally we compare the $K_1$-invariant part of the generalized characters \break 
$\Qcal^{\Phi,\rho}_{K_1\times K_2}(M\times \Xcal_P)$ and $\Qcal^{\Phi_1}_{K_1\times K_2}(M\times \Xcal_P)$.

\begin{prop}\label{prop:etape3}
There exists  $\epsilon>0$ such that 
\begin{equation}\label{eq:etape3}
\Qcal_{K_2}(M_P)-\left[\Qcal^{\Phi,\rho}_{K_1\times K_2}(M\times \Xcal_P)\right]^{K_1}=  
O((\epsilon/\rho)^{1/2})\quad {\rm in}\quad \Rfor(K_2)
\end{equation}
if $\rho>0$ is small enough.
\end{prop}

\medskip

If we sum the relations (\ref{eq:etape1}), (\ref{eq:etape2}) and (\ref{eq:etape3}) we get 
$$
\Qcal^{\Phi}_{K_2}(M) = \Qcal_{K_2}(M_P) + O(r_P)+ O((\epsilon/\rho)^{1/2})
$$
if $\rho$ is small enough. So Theorem \ref{theo:fondamental} follows by taking $(\epsilon/\rho)^{1/2}\geq r_P$.

\medskip

\subsection{Proof of Proposition \ref{prop:etape1}}

The cotangent bundle $\T^*K$ is identified with $K\times \kgot^*$. The data is then (see Section \ref{sec:TK}):

$\bullet$ the Liouville $1$-form $\lambda=\sum_j \omega_j\otimes E_j$. Here $(E_j)$ is a basis of $\kgot$ with dual basis 
$(E_j^*)$, and $\omega_j$ is the left invariant $1$-form on $K$ defined by 
$\omega_j(\frac{d}{dt} a\,e^{tX}\vert_{0})=\langle E_j^*,X\rangle$. 

$\bullet$ the symplectic form $\Omega:=-d\lambda$,

$\bullet$ the action of $K_1\times K_2$ on $K\times \kgot^*$ is  $(k_1,k_2)\cdot (a,\xi)= (k_2 a k_1^{-1}, k_1\cdot \xi)$,

$\bullet$ the moment map relative to the $K_1$-action is $\Phi_r(a,\xi)=-\xi$,

$\bullet$ the moment map relative to the $K_2$-action is $\Phi_l(a,\xi)=a\cdot\xi$.

\medskip

 We work now with the Hamiltonian action of $K_1\times K_2$ on $M\times\T^* K$ given by 
$$
(k_1,k_2)\cdot (m,a,\xi)= (k_1\cdot m,k_2 a k_1^{-1}, k_1\cdot \xi).
$$
The corresponding  moment map is $\Phi=(\Phi_1,\Phi_2)$: 
$\Phi_1(m,a,\xi)=\Phi_M(m)-\xi$ and $\Phi_2(m,a,\xi)=a\cdot\xi$.

Let $\clif_1$ be a symbol $\Thom(M,J_1)\otimes L$ attached to the prequantized Hamiltonian $K_1$-manifold $(M,\Omega)$. 
The cotangent bundle $\T^*K$ is prequantized by the trivial line bundle: let $\clif_2$ be the 
symbol $\Thom(\T^* K,J_2)$ attached to the prequantized Hamiltonian $K_1\times K_2$-manifold $\T^*K$. 
The product $\clif=\clif_1\odot\clif_2$ corresponds to the symbol $\Thom(N,J)\otimes L$ on $N=M\times \T^*K$.  

Let $\kappa_\rho$ be the Kirwan vector field associated to the map $\|\Phi\|^2_\rho: M\times \T^*K\to \R$.
We check that  $\|\Phi\|^2_\rho(m,k,\xi)=\|\Phi_M(m)-\xi\|^2 +\rho\|\xi\|^2$, and 
$$
\kappa_\rho(m,k,\xi)=\Big( \underbrace{(\Phi_M(m)-\xi)\cdot m}_{\kappa_I} \,;\, 
\underbrace{\wtde{\Phi}_M(m)-(1+\rho)\wtde{\xi}}_{\kappa_{I\!\!I,\rho}} \,;\, 
\underbrace{-[\wtde{\Phi}_M(m),\wtde{\xi}]}_{\kappa_{I\!\!I\!\!I}} \Big).
$$
Here $\T_{(m,k,\xi)}(M\times\T^*K)\simeq\T_{m} M\times \kgot\times \kgot$. We have 
\begin{eqnarray*}
\Cr(\|\Phi\|^2_\rho)&=&\{\kappa_\rho=0\}\\
&=&\bigcup_{\beta\in\Bcal} K_1\times K_2\cdot
\left[M^{\wtde{\beta}}\cap\Phi_M^{-1}(\beta)\times\{1\}\times\{\frac{\beta}{\rho+1}\}\right]
\end{eqnarray*}
where $\Bcal$ parametrizes $\Cr(\|\Phi_M\|^2)$. Hence one checks that the critical values of $\|\Phi\|^2_\rho$ are 
$\frac{\rho}{\rho+1}\|\beta\|^2,\beta\in\Bcal$.

Let $\clif^{\kappa_\rho}$ be the symbol $\clif$ pushed by the vector field $\kappa_\rho$: we have 
$$
\clif^{\kappa_\rho}(v; X ; Y)=\clif_1(v -\kappa_{I})\odot\clif_2(X-\kappa_{I\!\! I,\rho}\,;\,Y-\kappa_{I\!\!I\!\!I})
$$
for $(v; X; Y)\in \T_{(m,k,\xi)}(M\times\T^*K)\simeq\T_{m} M\times \kgot\times \kgot$.

For a real $R>0$ we define the open invariant subsets 
of $M\times\T^* K$
\begin{eqnarray*}
U_{R}&:=&\{\|\Phi\|^2_\rho< R\}\\
V_{R}&:=&\{\|\Phi_M\|^2< R\}\times \T^* K.
\end{eqnarray*}

By definition the generalized index $\Qcal^{\Phi,\rho}_{K_1\times K_2}(M\times \T^*K)$ is defined 
as the limit of the equivariant index 
$$
\Qcal^{\Phi,\rho}_{K_1\times K_2}(U_R):=\indice^{K_1\times K_2}_{\Ucal_R}(\clif^{\kappa_\rho}\vert_{U_R}), 
$$
when $R$ goes to infinity (and stays outside the critical values of $\|\Phi\|^2_\rho$).

In the other hand, when $R'$ is a regular value of $\|\Phi_M\|^2$, we see that the 
symbol $\clif_\rho\vert_{V_{R'}}$ is $K_1\times K_2$-transversally elliptic. Let 
\begin{equation}\label{eq:indice-V-R-prime}
\indice^{K_1\times K_2}_{V_{R'}}(\clif^{\kappa_\rho}\vert_{V_{R'}}) 
\end{equation}
be its equivariant index. Notice that the index map is well-defined on $V_{R}=\{\|\Phi_M\|^2< R\}\times \T^* K$ 
since $\T^*K$ can be seen as a open subset of a compact manifold.

It is easy to check that for any $R>0$ there exists $R'> R$ such that 
$U_R\subset V_R'$. It implies that $\Qcal^{\Phi,\rho}_{K_1\times K_2}(M\times \T^*K)$ 
is also defined as the limit of (\ref{eq:indice-V-R-prime}) when $R'$ goes to infinity.

We look now to the deformation $\kappa_\rho(s)=(\kappa_I^s; \kappa^s_{I\!\!I,\rho}; s\kappa_{I\!\!I\!\!I})$, $s\in[0,1]$ where
$$
\kappa_I^s(m,\xi)=(\Phi_M(m)-s\xi)\cdot m\quad {\rm and}\quad 
\kappa^s_{I\!\!I,\rho}(m,\xi)=s\wtde{\Phi}_M(m)-(1+s\rho)\wtde{\xi}.
$$

Let $\clif^{\kappa_\rho(s)}$ be the symbol $\clif$ pushed by the vector field $\kappa_\rho(s)$.

\begin{lem}Let $R'$ be a regular value of $\|\Phi_M\|^2$. 

$\bullet$ The familly $\clif^{\kappa_\rho(s)}\vert_{V_{R'}},\ s\in[0,1]$ 
defines an homotopy of $K_1\times K_2$-transversally elliptic symbols on $V_{R'}$.

$\bullet$ The $K_1$-invariant part of $\indice^{K_1\times K_2}_{\Vcal_{R'}}(\clif^{\kappa_\rho(0)}\vert_{V_{R'}})$
is equal to $\Qcal^{\Phi}_{K_2}(M_{< R'})$.
\end{lem}

\begin{proof} The first point follows from the fact that 
$\Char(\clif^{\kappa_\rho(s)}\vert_{V_{R'}})\cap \T_{K_1\times K_2}(V_{R'})$, which is equal to 
$$
\left\{
(m,k,\frac{s}{1+s\rho}\Phi_M(m)), \ k\in K\ {\rm and}\  m\in \Cr(\|\Phi_M\|^2)\cap\{\|\Phi_M\|^2< R'\}
\right\},
$$
stays in a compact set when $s\in[0,1]$. 

The symbol $\clif^{\kappa_\rho(0)}\vert_{V_{R'}}$ is equal to the product of 
the symbol $\clif_1^\kappa\vert_{M<R'}$, which is $K_1$-transversally elliptic, with the 
symbol
$$
\clif_2^{\kappa}(X;Y)=\clif_2(X+\xi; Y)
$$
which is a $K_2$-transversally elliptic on $\T^*K$. A basic computation 
done in section \ref{sec:TK-phi} gives that 
\begin{eqnarray*}
\indice^{K_1\times K_2}_{\T^*K}(\clif_2^{\kappa})&=&{\rm L}^2(K)\\
&=&\sum_{\mu\in\what{K}} (V_\mu^{K_1})^*\otimes V_\mu^{K_2}
\end{eqnarray*}
in $\Rfor(K_1\times K_2)$. Finally the``multiplicative property'' (see Theorem \ref{theo:multiplicative-property}) gives 
\begin{eqnarray*}
\indice^{K_1\times K_2}_{V_{R'}}(\clif^{\kappa_\rho(0)}\vert_{V_{R'}})&=&\indice^{K_1}_{M<R'}(\clif^\kappa_1\vert_{M<R'})
\otimes\indice^{K_1\times K_2}_{\T^*K}(\clif_2{\kappa})\\
&=& \sum_{\mu\in\what{K}} \Qcal^{\Phi}_{K_1}(M_{< R'})\otimes (V_\mu^{K_1})^*\otimes V_\mu^{K_2}
\end{eqnarray*}
Taking the $K_1$-invariant completes the proof of the second point.
\end{proof}

\medskip

Finally we have proved that the generalized character 
$[\indice^{K_1\times K_2}_{V_{R'}}(\clif^{\kappa_\rho}\vert_{V_{R'}})]^{K_1}$ 
is equal to $\Qcal^{\Phi}_{K_2}(M_{< R'})$. Taking the limit $R'\to\infty$ 
gives
\begin{eqnarray*}
\left[\Qcal^{\Phi,\rho}_{K_1\times K_2}(M\times \T^*K)\right]^{K_1}
&=&\lim_{R'\to\infty}\left[\indice^{K_1\times K_2}_{V_{R'}}(\clif^{\kappa_\rho}\vert_{V_{R'}})\right]^{K_1}\\
&=&\lim_{R'\to\infty}\Qcal^{\Phi}_{K_2}(M_{< R'})=\Qcal^{\Phi}_{K_2}(M).
\end{eqnarray*}

\subsection{Proof of Proposition \ref{prop:etape2}}

We work here with the Hamiltonian action of $K_1\times K_2$ on $M\times \Xcal_P$. The action is 
$(k_1,k_2)\cdot(m,x)=(k\cdot m,k_2\cdot x\cdot k_1^{-1})$
and  the corresponding moment map is $\Phi=(\Phi_1,\Phi_2)$ with  $\Phi_1(m,x)=\Phi_M(m)+\Phi_r(x)$ and $\Phi_2(m,x)=\Phi_l(x)$.
Let $\|(\xi_1,\xi_2)\|^2_\rho=\|\xi_1\|^2+\rho\|\xi_2\|^2$ be the Euclidean norm  $\kgot^*_1\times\kgot^*_2$ attached to $\rho>0$.

Let us consider the quantization of $M\times \Xcal_P$ via the map $\|\Phi\|^2_\rho$: 
$$
\Qcal^{\Phi,\rho}_{K_1\times K_2}(M\times \Xcal_P)\in \Rfor(K_1\times K_2)
$$
The critical set $\Cr(\|\Phi\|^2_\rho)$ admits the decomposition 
\begin{equation}\label{eq:crit-Phi-rho}
\Cr(\|\Phi\|^2_\rho)=\bigcup_{\gamma\in\Bcal_\rho}K_1\times K_2\cdot 
\Ccal_{\gamma}
\end{equation}
where $(m,x)\in \Ccal_{\gamma}$ if and only if  $\gamma=(\gamma_1,\gamma_2)$ with 
\begin{equation}\label{eq:C-gamma}
\begin{cases}
   \Phi_M(m)+\Phi_r(x)=\gamma_1\\
   \Phi_l(x)=\gamma_2  \\
  \wtde{\gamma}_1\cdot m=0\\
   \wtde{\gamma}_1\cdot_r x +\rho\,\wtde{\gamma}_2\cdot_l x=0.
\end{cases}
\end{equation}
We have 
\begin{equation}\label{eq:indice-M-X-P}
\Qcal^{\Phi,\rho}_{K_1\times K_2}(M\times \Xcal_P)=\sum_{\gamma\in\Bcal_\rho}
\Qcal^{\gamma,\rho}_{K_1\times K_2}(M\times \Xcal_P)
\end{equation}
where the generalized character $\Qcal^{\gamma,\rho}_{K_1\times K_2}(M\times \Xcal_P)$ 
is computed as an index of a transversally elliptic symbol in a neighborhood of 
$$
K_1\times K_2\cdot 
\Ccal_{\gamma}\subset M\times  \Phi^{-1}_l(K_2\cdot\gamma_2).
$$

Thanks to Theorem \ref{prop:Q-beta-support}  we know that the support of the generalized character 
$\Qcal^{\gamma,\rho}_{K_1\times K_2}(M\times \Xcal_P)$ is contained in 
$\{ (a,b)\in \what{K_1}\times\what{K_2}\ \vert\ \|a\|^2+\rho\| b\|^2\geq 
\|\gamma\|_\rho^2\}$. Hence
$$
{\rm support} \left(\left[\Qcal^{\gamma,\rho}_{K_1\times K_2}(M\times \Xcal_P))\right]^{K_1}\right)
\subset \Big\{ b\in \what{K_2}\ \vert\ \rho\| b\|^2\geq \|\gamma\|_\rho^2\Big\}
$$
Let $r_P=\inf_{\xi\in\partial P} \|\xi\|$. We know then that
$$
\left[\Qcal^{\Phi,\rho}_{K_1\times K_2}(M\times \Xcal_P)\right]^{K_1}
= \sum_{\stackrel{\gamma\in\Bcal_\rho}{\|\gamma\|^2_\rho< \rho r_P^2}} 
\left[\Qcal^{\gamma,\rho}_{K_1\times K_2}(M\times \Xcal_P)\right]^{K_1} + O(r_P).
$$

Let $R_P<\rho r_P^2$ be a regular value  of $\|\Phi\|^2_\rho : M\times \Xcal_P\to \R$  such that 
for all $\gamma\in \Bcal_\rho $ we have $\|\gamma\|^2_\rho< \rho r^2_P\Longleftrightarrow 
\|\gamma\|^2_\rho< R_P$. Then 
\begin{equation}\label{eq:decomp-M-X}
 \left[\Qcal^{\Phi,\rho}_{K_1\times K_2}(M\times \Xcal_P)\right]^{K_1}
= \left[\Qcal^{\Phi,\rho}_{K_1\times K_2}((M\times \Xcal_P)_{<R_P})\right]^{K_1} + O(r_P).
\end{equation}

\medskip

For the generalized index $\Qcal^{\Phi,\rho}_{K_1\times K_2}(M\times \T^*K)$ we have also a decomposition
$$
\Qcal^{\Phi,\rho}_{K_1\times K_2}(M\times \T^*K)=\sum_{\gamma\in\Bcal'_\rho}
\Qcal^{\gamma,\rho}_{K_1\times K_2}(M\times \T^*K)
$$
where $\Bcal'_\rho$ parametrizes the critical set of 
$\|\Phi\|_\rho^2:M\times \T^*K\to \R$. Like before we get 
\begin{equation}\label{eq:decomp-M-TK}
\left[\Qcal^{\Phi,\rho}_{K_1\times K_2}(M\times \T^*K)\right]^{K_1}
= \left[\Qcal^{\Phi,\rho}_{K_1\times K_2}((M\times \T^*K)_{<R'_P})\right]^{K_1} + O(r_P).
\end{equation}
Here $R'_P<\rho r_P^2$ is a regular value  of $\|\Phi\|^2_\rho:M\times \T^*K\to \R$ such that 
for all $\gamma\in \Bcal'_\rho $ we have $\|\gamma\|^2_\rho< \rho r^2_P\Longleftrightarrow 
\|\gamma\|^2_\rho< R'_P$.

\begin{lem}
We have 
\begin{equation}\label{eq:Q-Phi-R-P}
\Qcal^{\Phi,\rho}_{K_1\times K_2}((M\times \Xcal_P)_{<R_P})=\Qcal^{\Phi,\rho}_{K_1\times K_2}((M\times \T^*K)_{< R'_P}) .
\end{equation}
\end{lem}

\begin{proof} The Lemma will follow from Proposition \ref{prop:psi-invariance}. We take here $V'=M\times \Xcal_P^o$, 
$V=M\times K\times \Ucal_P\subset M\times \T^* K$ and the equivariant diffeomorphism $\Psi: V\to V'$ is equal to 
${\rm Id}\times \Upsilon$ where $\Upsilon$ was introduced in Section \ref{subsec:wonderfull}. Note that $\Psi$ satisfies points  
$(1)-(3)$ of Proposition \ref{prop:psi-invariance}.

Note that $\|\Phi(m,x)\|_\rho^2<\rho r_P^2$ implies that $\|\Phi_l(x)\|<r_P$ and then $x\in \Xcal_P^o$. Hence the open 
subset $U':=(M\times \Xcal_P)_{<R_P}$ is contained in $V'=M\times \Xcal_P^o$. In the same way the open subset 
$U:=(M\times \T^* K)_{<R'_P}$ is contained in $V$. We have $\Psi(U)=U'$ if $R_P=R'_P$. 

We have proved that (\ref{eq:Q-Phi-R-P}) is a consequence of Proposition \ref{prop:psi-invariance}.
\end{proof}

Finally, if we take the difference between (\ref{eq:decomp-M-X}) and (\ref{eq:decomp-M-TK}), we get
$$
\left[\Qcal^{\Phi,\rho}_{K_1\times K_2}(M\times \Xcal_P)\right]^{K_1}-
\left[\Qcal^{\Phi,\rho}_{K_1\times K_2}(M\times \T^*K)\right]^{K_1}
=O(r_P).
$$
which is the relation of Proposition \ref{prop:etape2}.

\subsection{Proof of Proposition \ref{prop:etape3}}

Here we want to compare the $K_1$-invariant part of the characters  
$\Qcal^{\Phi,\rho}_{K_1\times K_2}(M\times \Xcal_P)$ and $\Qcal^{\Phi_1}_{K_1\times K_2}(M\times \Xcal_P)$.

We know after Theorem \ref{theo:Q-sing} that 
\begin{eqnarray*}
\Qcal_{K_2}(M_P)&=& \left[\Qcal^{\Phi_1}_{K_1\times K_2}(M\times \Xcal_P)\right]^{K_1}\\
&=& \left[\Qcal^{\Phi_1}_{K_1\times K_2}(U_\epsilon)\right]^{K_1}
\end{eqnarray*}
when $\epsilon>0$ is any regular value of $\|\Phi_1\|^2$, and 
$U_\epsilon:=\{\|\Phi_1\|^2<\epsilon\}\subset M\times \Xcal_P$. 

In this section we fix once for all $\epsilon >0$ small enough so 
that 
\begin{equation}\label{eq:epsilon-c}
\Cr(\|\Phi_1\|^2)\cap\{\|\Phi_1\|^2\leq \epsilon\} =\{\Phi_1=0\}.
\end{equation}  

Let $\clif_1$ be the symbol $\Thom(M,J_1)\otimes L$ attached to the prequantized Hamiltonian $K_1$-manifold $(M,\Omega)$. 
Let $\clif_3$ be the symbol $\Thom(\Xcal_P,J_3)\otimes L_P$ attached to the prequantized Hamiltonian $K_1\times K_2$-manifold 
$\Xcal_P$. The product $\clif=\clif_1\odot\clif_3$ corresponds to the symbol $\Thom(N,J)\otimes L$ on $N=M\times \Xcal_P$.

Let $\kappa_0$ and $\kappa_\rho$ be the Kirwan vector fields associated to the functions $\|\Phi_1\|^2$ and 
$\|\Phi\|^2_\rho$ on $M\times \Xcal_P$:
$$
\kappa_0(m,x)=\Big( \underbrace{\Phi_1(m,x)\cdot m}_{\kappa_I} \,;\, 
\underbrace{\Phi_1(m,x)\cdot_r x}_{\kappa_{I\!\!I}} \Big),\quad \kappa_\rho(m,x)=\kappa^0(m,x)+
\rho\, (0, \underbrace{\Phi_l(x)\cdot_l x}_{\kappa_{I\!\!I\!\!I}}).
$$

Let $\clif^{\kappa_\rho}$ be the symbol $\clif$ pushed by the vector field $\kappa_\rho$: we have 
$$
\clif^{\kappa_\rho}(v; \eta)= \clif_1(v -\kappa_{I})\odot\clif_3(\eta-\kappa_{I\!\! I}-\rho\kappa_{I\!\!I\!\!I})
$$
for $(v; \eta)\in \T_{(m,x)}(M\times\Xcal_P)$.

The character $\Qcal^{\Phi_1}_{K_1\times K_2}(U_\epsilon)$ is given by the index of 
the $K_1$-transversally elliptic symbol $\clif^{\kappa_0}\vert_{U_\epsilon}$.  The character 
$\Qcal^{\Phi,\rho}_{K_1\times K_2}(M\times \Xcal_P)$ is given by the index of 
the $K_1\times K_2$-transversally elliptic symbol $\clif^{\kappa_\rho}$.


\begin{lem}\label{lem:crit-phi-rho}
$\bullet$ There exists $\rho(\epsilon)>0$  such that
$$
\Cr(\|\Phi\|^2_\rho)\bigcap \left\{\|\Phi_1\|^2\leq \epsilon\right\}\subset \left\{\|\Phi_1\|^2\leq \frac{\epsilon}{2}\right\}
$$
for any $0\leq \rho\leq \rho(\epsilon)$.
\end{lem}

\begin{proof}With the help of Riemannian metrics on $M$ and $\Xcal_P$ we define 
\begin{eqnarray*}
a(\epsilon)&:=&\inf_{\epsilon/2\leq \|\Phi_1(m,x)\|\leq \epsilon} \|\kappa^0(m,x)\|\\
b&:=&\sup_{x\in \Xcal_P} \|\Phi_l(x)\cdot_l x\|.
\end{eqnarray*}
We have $a(\epsilon)>0$ thanks to (\ref{eq:epsilon-c}), and $b<\infty$ since $\Xcal_P$ is compact. It is 
now easy to check that $\{\kappa_\rho=0\}\cap\{\epsilon/2\leq\|\Phi_1\|^2\leq \epsilon\}=\emptyset$ if 
$0\leq \rho < \frac{a(\epsilon)}{b}$.

\end{proof}

The symbols $\clif^{\kappa_\rho}\vert_{U_\epsilon}$, $\rho\in[0,\rho(c)]$ are $K_1\times K_2$-transversally 
elliptic, and they define the same class in $\K_{K_1\times K_2}(\T_{K_1\times K_2}U_\epsilon)$. Hence 
$\Qcal_{K_2}(M_P)$ can be computed as the $K_1$-invariant part of 
$$
\Qcal^{\Phi,\rho}_{K_1\times K_2}(U_\epsilon):=\indice^{U_\epsilon}_{K_1\times K_2}(\clif^{\kappa_\rho}\vert_{U_\epsilon})\in \Rfor(K_1\times K_2) 
$$
for $\rho\in[0,\rho(\epsilon)]$. 

A component $K_1\times K_2\cdot \Ccal_{\gamma}$ of $\Cr(\|\Phi\|^2_\rho)$ is contained in $U_\epsilon$ if and only 
$\|\gamma_1\|< \epsilon$ : hence the decomposition (\ref{eq:indice-M-X-P}) for the character 
$\Qcal^{\Phi,\rho}_{K_1\times K_2}(M\times \Xcal_P)$ gives
$$
\Qcal^{\Phi,\rho}_{K_1\times K_2}(M\times \Xcal_P)=\Qcal^{\Phi,\rho}_{K_1\times K_2}(\Ucal_\epsilon)+
\sum_{\stackrel{\gamma\in\Bcal_\rho}{\|\gamma_1\|^2\geq \epsilon}}
\Qcal^{\gamma,\rho}_{K_1\times K_2}(M\times \Xcal_P).
$$
where
$$
\Qcal^{\Phi,\rho}_{K_1\times K_2}(\Ucal_\epsilon)=\sum_{\stackrel{\gamma\in\Bcal_\rho}{\|\gamma_1\|^2< \epsilon}}
\Qcal^{\gamma,\rho}_{K_1\times K_2}(M\times \Xcal_P).
$$

Taking the $K_1$-invariant gives
\begin{equation}\label{eq:indice-U-c}
[\Qcal^{\Phi,\rho}_{K_1\times K_2}(M\times \Xcal_P)]^{K_1}=\Qcal_{K_2}(M_P)+
\sum_{\stackrel{\gamma\in\Bcal_\rho}{\|\gamma_1\|^2\geq \epsilon}}
[\Qcal^{\gamma,\rho}_{K_1\times K_2}(M\times \Xcal_P)]^{K_1}.
\end{equation}
In general we know that the support of the generalized character \break
$[\Qcal^{\gamma,\rho}_{K_1\times K_2}(M\times \Xcal_P))]^{K_1}$ is included in 
$\{ b\in \what{K_2}\ \vert\ \rho\| b\|^2\geq \|\gamma_1\|^2+\rho\|\gamma_2\|^2\}$. 
When $\|\gamma_1\|^2\geq \epsilon$ we have then that the support of 
$[\Qcal^{\gamma,\rho}_{K_1\times K_2}(M\times \Xcal_P))]^{K_1}$ is contained in 
$\{ b\in \what{K_2}\ \vert\ \rho\| b\|^2\geq \epsilon\}$.

Finally (\ref{eq:indice-U-c}) imposes that 
$$
[\Qcal^{\Phi,\rho}_{K_1\times K_2}(M\times \Xcal_P)]^{K_1}=\Qcal_{K_2}(M_P)+O((\epsilon/\rho)^{1/2}).
$$
when $0<\rho\leq \rho(\epsilon)$, which is the precise content of Proposition \ref{prop:etape3}.

\section{Other properties of $\Qcal^\Phi$}\label{sec:prop-Q-Phi}

Let $(M,\omega,\Phi)$ be a proper Hamitlonian $K$-manifold which is prequantized by a line bundle $L$. 
The character $\Qcal^\Phi_K(M)$ is computed by means of a scalar product on $\kgot^*$. 
The fact that $\Qcal^\Phi_K(M)=\qfor_K(M)$ gives the following 

\begin{prop}\label{prop:norm-independant}
The character $\Qcal^\Phi_K(M)$  does not depend of the choice of a scalar product on $\kgot^*$
\end{prop}

In this section we work in the setting where $K=K_1\times K_2$.  Let $\Phi_1$ be the 
moment map relative to the $K_1$-action.

\subsection{$\Phi_1$ is proper}

In this subsection we suppose that 
the moment map $\Phi_1$ relative to the $K_1$-action is {\em proper}.  We fix an invariant Euclidean norm 
$\|\bullet\|^2$ on $\kgot$ in such a way  that $\kgot_1= \kgot_2^\perp$. 

Let us ``quantize'' $(M,\Omega)$ via the invariant proper function $\|\Phi_1\|^2$: let 
$$
\Qcal^{\Phi_1}_{K_1\times K_2}(M)\in\Rfor(K_1\times K_2)
$$
be the corresponding generalized character.

\begin{theo}\label{theo:Phi=Phi_1}
We have 
\begin{equation}\label{eq:Phi=Phi_1}
\Qcal^\Phi_{K_1\times K_2}(M)=\Qcal^{\Phi_1}_{K_1\times K_2}(M) \quad {\rm in}\quad \Rfor(K_1\times K_2).
\end{equation}
\end{theo}

\medskip

\begin{proof} On $\kgot=\kgot_1\oplus \kgot_2$ we may consider the family of invariant  Euclidean norms:
$\|X_1\oplus X_2\|^2_\rho=\|X_1\|^2+\rho\|X_2\|^2$ for $X_j\in\kgot_j$. Let 
$$
\Qcal^{\Phi,\rho}_{K_1\times K_2}(M)\in\Rfor(K_1\times K_2)
$$
be the quantization of $M$ computed via the map $\|\Phi\|^2_\rho= \|\Phi_1\|^2+\rho \|\Phi_2\|^2$. By definition,
$\Qcal^{\Phi_1}_{K_1\times K_2}(M)$ is equal to $\Qcal^{\Phi,0}_{K_1\times K_2}(M)$, and we know after 
Proposition \ref{prop:norm-independant} that  $\Qcal^\Phi_{K_1\times K_2}(M)$ coincides with the generalized character 
$\Qcal^{\Phi,\rho}_{K_1\times K_2}(M)\in\Rfor(K)$ for any $\rho>0$. 

Let us prove that prove that $\Qcal^{\Phi,\rho}_{K_1\times K_2}(M)=\Qcal^{\Phi_1}_{K_1\times K_2}(M)$. 
We denote $O(r)\in\Rfor(K_1\times K_2)$ any generalized character supported outside the 
ball $$\{\xi\in\tgot^*_1\times\tgot^*_2\ \vert\ \|\xi_1\|^2+\|\xi_2\|^2< r^2\}.$$
And we denote $O_1(r)\in\Rfor(K_1\times K_2)$ any generalized character supported outside the 
$$\{\xi\in\tgot^*_1\times\tgot^*_2\ \vert\ \|\xi_1\|< r\}.$$

Let $R_1>0$ be a regular value of $\|\Phi_1\|^2$: the open subset $\{\|\Phi_1\|^2< R_1\}$ is denoted $M_{<R_1}$. 
We know that 
$$
\Qcal^{\Phi_1}_{K_1\times K_2}(M)= \Qcal^{\Phi_1}_{K_1\times K_2}(M_{<R_1}) +O_1(\sqrt{R_1}).
$$
Like in the Lemma \ref{lem:crit-phi-rho}, we know that 
\begin{equation}\label{eq:crit-Phi-rho-R1}
\Cr(\|\Phi\|_\rho^2)\cap \{\|\Phi_1\|^2=R_1\}=\emptyset.
\end{equation}
for $\rho\geq 0$ small enough. The identity (\ref{eq:crit-Phi-rho-R1}) first implies that 
\begin{eqnarray*}
\Qcal^{\Phi,\rho}_{K_1\times K_2}(M)&=&\sum_{\stackrel{\gamma\in\Bcal_\rho}{\|\gamma_1\|^2< R_1}}
\Qcal^{\gamma,\rho}_{K_1\times K_2}(M) +
\sum_{\stackrel{\gamma\in\Bcal_\rho}{\|\gamma_1\|^2> R_1}}
\Qcal^{\gamma,\rho}_{K_1\times K_2}(M)\\
&=&\Qcal^{\Phi,\rho}_{K_1\times K_2}(M_{< R_1})+ O(\sqrt{R_1}).
\end{eqnarray*}
In the second equality we have used that $\Qcal^{\gamma,\rho}_{K_1\times K_2}(M)=O(\sqrt{R_1})$ when 
$\|\gamma_1\|^2> R_1$ since the ball 
$\left\{(\xi_1,\xi_2)\in \tgot^*_1\times\tgot^*_2\ \vert\ \|\xi_1\|^2+\|\xi_2\|^2< R_1\right\}$ is contained 
in 
$$
\left\{(\xi_1,\xi_2)\in \tgot^*_1\times\tgot^*_2\ \vert\ \|(\xi_1,\xi_2)\|^2_\rho< \|(\gamma_1,\gamma_2)\|^2_\rho\right\}.
$$
The identity (\ref{eq:crit-Phi-rho-R1}) shows also that the symbol $\clif^{\kappa_\rho}\vert_{M<R_1}$ are homotopic 
for $\rho\geq 0$ small enough. Hence 
$$
\Qcal^{\Phi,\rho}_{K_1\times K_2}(M_{< R_1})=\Qcal^{\Phi_1}_{K_1\times K_2}(M_{< R_1})
$$ 

We get finally that $\Qcal^{\Phi,\rho}_{K_1\times K_2}(M)-\Qcal^{\Phi_1}_{K_1\times K_2}(M)= 
O(\sqrt{R_1})+ O_1(\sqrt{R_1})$ for any regular value $R_1$ of $\|\Phi_1\|^2$. We have proved 
that $\Qcal^{\Phi,\rho}_{K_1\times K_2}(M)- \Qcal^{\Phi_1}_{K_1\times K_2}(M)=0$. 
\end{proof}

\medskip

Let us explain how Theorem \ref{theo:Phi=Phi_1} contains the identity that we called 
``{\em quantization commutes with reduction in the singular setting}''  in \cite{pep-formal}. By 
definition the $K_1$-invariant part of the right hand side of (\ref{eq:Phi=Phi_1}) is equal to 
the geometric quantization of the (possibly singular) compact Hamiltonian $K_2$-manifold  
$$
M{/\!\!/}_{0}K_1:=\Phi_1^{-1}(0)/K_1.
$$
Using now the fact that  the left hand side of (\ref{eq:Phi=Phi_1}) is equal to $\qfor_{K_1\times K_2}(M)$, we 
see  that the multiplicity of $V_\mu^{K_2}$ in $\Qcal_{K_2}(M{/\!\!/}_{0}K_1)$ is equal to the 
geometric quantization of the (possibly singular) compact manifold  
$$
M\times \overline{K_2\cdot\mu}{/\!\!/}_{(0,\mu)}K_1\times K_2.
$$

\subsection{The symplectic reduction $M{/\!\!/}_{0}K_1$ is smooth}\label{subsec:smooth-K-1-reduction}

Let $(M,\Omega)$ be an Hamiltonian $K_1\times K_2$-manifold with a proper moment map $\Phi=(\Phi_1,\Phi_2)$.  
In this section we suppose that $0$ is a regular value of $\Phi_1$ and that $K_1$ acts freely on $\Phi_1^{-1}(0)$. 
We work then with the (smooth) Hamiltonian $K_2$-manifold 
$$
N:=\Phi_1^{-1}(0)/K_1. 
$$
We still denote  by $\Phi_2:N\to\kgot^*_2$ the moment map relative to the $K_2$-action: note that this map is proper.  
Hence we can quantize the $K_2$-action on $N$ via the map $\Phi_2$. Let 
$\Qcal^{\Phi_2}_{K_2}(N)\in\Rfor(K_2)$ be the corresponding character.

\begin{prop}
We have 
\begin{equation}\label{eq:Q-Phi-K2-invariant}
\left[ \Qcal^{\Phi}_{K_1\times K_2}(M) \right]^{K_1}=\Qcal^{\Phi_2}_{K_2}(N)\quad {\rm in}\quad\Rfor(K_2).
\end{equation}
\end{prop}

\begin{proof} When $\Phi_1$ is proper, the manifold $N$ is compact. Then the right hand side of (\ref{eq:Q-Phi-K2-invariant}) is equal to 
$\Qcal_{K_2}(N)$, and we know from Theorem \ref{theo:Phi=Phi_1} that the left 
hand side of (\ref{eq:Q-Phi-K2-invariant}) is equal to $[ \Qcal^{\Phi_1}_{K_1\times K_2}(M) ]^{K_1}$. In this 
case (\ref{eq:Q-Phi-K2-invariant}) becomes $[ \Qcal^{\Phi_1}_{K_1\times K_2}(M) ]^{K_1}=\Qcal_{K_2}(M{/\!\!/}_{0}K_1)$
which is the content of Theorem \ref{theo:Q-R-lisse}.

Let us consider the general case where $\Phi_1$ is not proper. Thanks to Theorem \ref{theo:intro} one knows that the multiplicities of 
$V_\mu^{K_2}$ in $\left[ \Qcal^{\Phi}_{K_1\times K_2}(M) \right]^{K_1}$ and $\Qcal^{\Phi_2}_{K_2}(N)$ 
are respectively equal to the quantization of the (possibly 
singular) symplectic reductions 
$$
\Mcal_{\mu}:=M\times \overline{K_2\cdot\mu}{/\!\!/}_{(0,0)}K_1\times K_2.
$$
and
$$
\Mcal'_{\mu}:=N\times \overline{K_2\cdot\mu}{/\!\!/}_{0}K_2,\quad {\rm with}\quad N=M{/\!\!/}_{0}K_1.
$$

Note that $\Mcal_{\mu}$ and $\Mcal_{\mu}'$ coincide as symplectic reduced space. Let us prove that their geometric quantization
are identical also. The proof will be done for $\mu=0$: the other case follows from the shifting trick.

Let $\clif$ be the $K_1\times K_2$-equivariant symbol $\Thom (M,J)\otimes L_M$. Let $\kappa$ be the Kirwan vector field attached 
to the moment map $\Phi=(\Phi_1,\Phi_2)$. Let $\clif^\kappa$ be the symbol $\clif$ pushed by $\kappa$. Let us denote 
$M_{<\epsilon}$ the open subset $\{\|\Phi\|^2<\epsilon\}$.  For $\epsilon>0$  small enough, the symbol $\clif^\kappa\vert_{M_{<\epsilon}}$ 
is $K_1\times K_2$-transversally elliptic, and $\Qcal(\Mcal_{0})$ is the $K_1\times K_2$-invariant part of 
$\indice_{M_{<\epsilon}}^{K_1\times K_2}(\clif^\kappa\vert_{M_{<\epsilon}})$.

Let $\clif_2$ be the $K_2$-equivariant symbol $\Thom (N,J)\otimes L_N$. Let $\kappa_2$ be the Kirwan vector field attached 
to the moment map $\Phi_2$. Let $\clif_2^{\kappa_2}$ be the symbol $\clif_2$ pushed by $\kappa_2$. Let us denote 
$N_{<\epsilon}$ the open subset $\{\|\Phi_2\|^2<\epsilon\}$.  For $\epsilon>0$  small enough, the symbol 
$\clif_2^{\kappa_2}\vert_{N_{<\epsilon}}$ 
is $K_2$-transversally elliptic, and $\Qcal(\Mcal'_{0})$ is the $K_2$-invariant part of 
$\indice_{N_{<\epsilon}}^{K_2}(\clif_2^{\kappa_2}\vert_{N_{<\epsilon}})$.

Our proof follows from the comparison of the classes
$$
\left[\clif^\kappa\vert_{M_{<\epsilon}}\right]\in\K_{K_1\times K_2}(\T_{K_1\times K_2}M_{<\epsilon})
$$
and 
$$
\left[\clif_2^{\kappa_2}\vert_{N_{<\epsilon}}\right]\in\K_{K_2}(\T_{K_2} N_{<\epsilon})
$$
A neighborhood of the smooth submanifold $Z:=\Phi_1^{-1}(0)$ in $M$ is diffeomorphic to a neighborhood of the 
$0$-section of the bundle $Z\times \kgot_1^*\to Z$. Let $Z_{<\epsilon}=Z\cap M_{<\epsilon}$ so that
$N_{<\epsilon}=Z_{<\epsilon}/K_1$. Hence $\left[\clif^\kappa\vert_{M_{<\epsilon}}\right]$ can be seen naturally a class in the K-group 
$\K_{K_1\times K_2}(\T_{K_1\times K_2}(Z_{<\epsilon}\times\kgot_1^*))$.

Following Atiyah \cite{Atiyah74}[Theorem 4.3], the inclusion map
$j:Z_{<\epsilon}\croc Z_{<\epsilon}\times \kgot^{*}_1$ induces the Thom isomorphism
$$
j_{!}:\K_{K_1\times K_2}(\T_{K_1\times K_2}Z_{<\epsilon})\longrightarrow 
\K_{K_1\times K_2}(\T_{K_1\times K_2}(Z_{<\epsilon}\times \kgot_1^{*})),
$$ 
with the commutative diagram
\begin{equation}\label{j.point.bis}
\xymatrix@C=2cm{
 \K_{K_1\times K_2}(\T_{K_1\times K_2}Z_{<\epsilon})\ar[r]^{j_{!}} \ar[dr]_{\indice^{K_1\times K_2}_{Z_{<\epsilon}}} & 
 \K_{K_1\times K_2}(\T_{K_1\times K_2}( Z_{<\epsilon}\times \kgot_1^{*})) 
 \ar[d]^{\indice_{Z_{<\epsilon}\times \kgot_1^{*}}^{K_1\times K_2}}\\
     & R^{-\infty}(K_1\times K_2)}.
\end{equation} 

Let $\pi_1: Z_{<\epsilon}\to N_{<\epsilon}$ be the quotient relative to the free action of $K_1$. The corresponding 
isomorphism 
$$
\pi_1^*: \K_{K_2}(\T_{K_2}N_{<\epsilon}) \longrightarrow\K_{K_1\times K_2}(\T_{K_1\times K_2}Z_{<\epsilon})
$$
satisfies the following rule :
\begin{equation}\label{eq:indice-pi-1}
\left[\indice^{Z_{<\epsilon}}_{K_1\times K_2} ( \pi_1^*\theta)\right]^{K_1}=
\indice^{N_{<\epsilon}}_{K_2} (\theta)
\end{equation}
for any $\theta\in \K_{K_2}(\T_{K_2}N_{<\epsilon})$.

\begin{lem}[\cite{pep-RR}]\label{lem:induction-Z}
  We have 
  $$
  j_{!}\circ\pi_1^{*}\Big(\left[\clif_2^{\kappa_2}\vert_{N_{<\epsilon'}}\right]\Big)=\left[\clif^\kappa\vert_{M_{<\epsilon}}\right]
  $$
  in $\K_{K_1\times K_2}(\T_{K_1\times K_2}( Z_{<\epsilon}\times \kgot_1^{*}))$.
\end{lem} 

\begin{proof} This Lemma is proven in \cite{pep-RR}[Section 6.2] when the group $K_2$ is trivial. It is easy to check that 
the proof extends naturally to our setting.
\end{proof}

If one uses Lemma \ref{lem:induction-Z} together with  (\ref{j.point.bis}) and (\ref{eq:indice-pi-1}), we get that 
\begin{eqnarray*}
\Qcal(\Mcal_0)&=&\left[\indice_{Z_{<\epsilon}\times \kgot_1^{*}}^{K_1\times K_2}
(\clif^\kappa\vert_{M_{<\epsilon}})\right]^{K_1\times K_2}\\
&=&\left[\indice_{N_{<\epsilon}}^{K_2}(\clif_2^{\kappa_2}\vert_{N_{<\epsilon}})\right]^{K_2}= \Qcal(\Mcal'_0).
\end{eqnarray*}

\end{proof}

\bigskip

\section{Example: the cotangent bundle of an orbit}\label{sec:quant-K/H}

\subsection{The formal quantization of $\T^*K$.}\label{sec:TK}

Let $K$ be a compact connected Lie group equipped with the action of two copies of $K$: 
$(k_1,k_2)\cdot a = k_2 a k_1^{-1}$. Then we have a Hamiltonian action of $K_1\times K_2$ on the cotangent bundle 
$\T^*K$. In this section, we check that each formal geometric quantization of $\T^*K$,
$\qfor_{K_1\times K_2}(\T^*K)$ and $\Qcal^\Phi_{K_1\times K_2}(\T^*K)$, are both equal to 
the $K_1\times K_2$-module $L^2(K)$.

The tangent bundle $\T K$ is identified with $K\times\kgot$ through the right translations: to 
$(a,X)\in K\times\kgot$ we associate $\frac{d}{dt}a e^{tX}\vert_0 $. The action of 
$K_1\times K_2$ on the cotangent bundle $\T^*K\simeq K\times \kgot^*$ is then 
$$
(k_1,k_2)\cdot( a,\xi)= (k_2ak_1^{-1},k_1\cdot\xi).
$$

The symplectic form on $T^*K$ is $\Omega:=-d\lambda$, where $\lambda$ is the 
Liouville $1$-form. Let us compute these two form in coordinates. The tangent 
bundle of $\T^*K\simeq K\times \kgot^*$ is identified with $\T^*K\times \kgot\times\kgot^*$: 
for each $(a,\xi)\in \T^*K$, we have a two form $\Omega_{(a,\xi)}$  on $\kgot\times\kgot^*$. 
A direct computation gives 
$$\Omega_{(a,\xi)}(X,Y)=\langle \xi,[X,Y]\rangle,\quad \Omega_{(a,\xi)}(\eta,\eta')=0, \quad 
\Omega_{(a,\xi)}(X,\eta)=\langle \eta, X\rangle
$$
for $X,Y\in\kgot$ and $\eta,\eta'\in\kgot^*$. So $\Omega_{(a,\xi)}=\Omega_0 + \pi_\xi$
where $\Omega_0$ is the canonical (constant) symplectic form on $\kgot\times\kgot^*$ and  
$\pi_\xi$  is the closed two form on $\kgot$ defined by $\pi_\xi(X,Y)=\langle \xi,[X,Y]\rangle$. 

If we identify $\kgot\simeq\kgot^*$ through an invariant Euclidean norm, the symplectic structure 
on $\T_{(a,\xi)}(\T^*K)\simeq \kgot\times\kgot^*$
is given by a skew-symmetric matrix
$$A_\xi:=\left(\begin{array}{cc}
\ad(\xi) & I_n \\
-I_n& 0
\end{array}\right).
$$
We will work with the following compatible almost complex structure  on the tangent bundle of $\T^*K$  :
$J_\xi= -A_\xi (-A^2_\xi)^{-1/2}$.
When $\xi= 0$, the complex structure $J_0$ on $\kgot\times\kgot^*$ is defined by the matrix 
$$J_0:=\left(\begin{array}{cc}
0 &- I_n \\
I_n& 0
\end{array}\right).
$$
Hence the complex $K$-module $(\kgot\times\kgot^*, J_0)$ is naturally identified with the complexification 
$\kgot_\C$ of $\kgot$.

One checks easily that the moment map relative to the $K_1\times K_2$-action is the {\em proper} 
map $\Phi: \T^*K\to \kgot_1^*\times \kgot_2^*$ defined by $\Phi(a,\xi)=(-\xi,a\cdot\xi)$. 

Here the symplectic manifold $\T^*K$ is prequantized by the trivial line bundle.

\subsubsection{Computation of $\qfor_{K_1\times K_2}(\T^*K)$.}\label{sec:TK-qfor}

Let $\Ocal_1\times\Ocal_2$ be a coadjoint orbit of $K_1\times K_2$ in $\kgot_1^*\times \kgot_2^*$. 
One checks that 
\begin{equation}
\Phi^{-1}(\Ocal_1\times\Ocal_2)=
\begin{cases}
   \emptyset\quad {\rm if}\ \Ocal_1\neq -\Ocal_2\\
   {\rm a} \ K_1\times K_2-{\rm orbit} \quad {\rm if}\ \Ocal_1= -\Ocal_2.
\end{cases}
\end{equation}

We knows that the stabiliser subgroup $K_\xi$ of an element $\xi\in\kgot^*$ is connected. Then the stabilizer subgroup 
$(K_1\times K_2)_{(a,\xi)}=\{(k_1,ak_1 a^{-1}),\ k_1\in K_\xi\}$ is also connected. 

Let $(\T^*K)_{(\mu,\lambda)}$ be the symplectic reduction of $\T^*K$ at the level $(\mu,\lambda)\in\what{K}^2$. For any $\mu\in\what{K}$, 
we define $\mu^*\in\what{K}$ by the relation $-K\cdot\mu=K\cdot\mu^*$: note that $V_{\mu^*}^K\simeq (V_\mu^K)^*$. 
If one uses Theorem \ref{theo:Q-point}, one has 
\begin{equation}
\Qcal((\T^*K)_{(\mu,\lambda)})=
\begin{cases}
  0\quad {\rm if}\ \lambda\neq \mu^*\\
   1\quad {\rm if}\ \lambda= \mu^*.
\end{cases}
\end{equation}

Finally
\begin{eqnarray*}
\qfor_{K_1\times K_2}(\T^*K)
&=&\sum_{(\mu,\lambda)\in \what{K}\times\what{K}}
\Qcal\left((\T^*K)_{(\mu,\lambda)}\right) V_{\mu}^{K_1}\otimes V_{\lambda}^{K_2}\\
&=& \sum_{\mu\in \what{K}} V_{\mu}^{K_1}\otimes (V_{\mu}^{K_2})^*={\rm L}^2(K).
\end{eqnarray*}

\subsubsection{Computation of $\Qcal^\Phi_{K_1\times K_2}(\T^*K)$.}\label{sec:TK-phi}

The Kirwan vector field on $\T^*K$ is 
$$
\kappa(a,\xi)=-2\xi\in\kgot_\C.
$$ 
Let $\clif^\kappa$ be the symbol $\Thom(\T^*K,J)$ pushed by the vector field 
$\frac{1}{2}\kappa$. At each $(a,\xi)\in \T^*K$, the map $\clif^\kappa_{(a,\xi)}(X\oplus\eta)$ from 
$\wedge_{ J_\xi}^{even}(\kgot\times\kgot^*)$ to $\wedge_{ J_\xi}^{odd}(\kgot\times\kgot^*)$ is equal 
to the clifford map $\clif(X+\xi\oplus\eta)$. Note that $\clif^\kappa$ is a $K_2$-transversally elliptic symbol
on $\T^*K$: we have $\Char(\clif^\kappa)\cap \T_{K_2}(\T^*K)=\{(1,0)\}$. We will now compute the equivariant index 
of $\clif^\kappa$.

First we consider the homotopy $t\in [0,1]\to J_{t\xi}$ of symplectic structure on $\T^*K$. Let $\tilde\clif^\kappa$ be the 
symbol acting on $\wedge_{ J_0}^{\bullet}(\kgot\times\kgot^*)=\wedge_\C^{\bullet} \kgot_\C$. Proposition 
\ref{prop:psi-invariance} shows that  the symbols $\clif^\kappa$ and $\tilde\clif^\kappa$ define the same class in 
$K_{K_1\times K_2}(\T_{K_2}(\T^*K))$.

The projection $\pi:\T^*K\to\kgot^*$ corresponds to the quotient map relative to the free action of $K_2$. At the 
level of $K$-groups we get an isomorphism
$$
\pi_*:K_{K_1\times K_2}(\T_{K_2}(\T^*K))\to K_{K_1}(\T\kgot^*).
$$
Atiyah \cite{Atiyah74} proves that
$$
\indice^{\T^*K}_{K_1\times K_2}(\sigma)=\sum_{\mu\in\what{K}}
\indice_{K_1}^{\kgot^*}\left(\pi_*(\sigma\otimes V_{\mu}^{K_2})\right)
\otimes (V_{\mu}^{K_2})^*
$$
for any class $\sigma\in K_{K_1\times K_2}(\T_{K_2}(\T^*K))$. In our case the symbol 
$\pi^*(\tilde\clif^\kappa)$ is equal to the 
Bott symbol ${\rm Bott}(\kgot^*)$, and for any $K_2$-module $E_2$ we have 
$$
\pi_*(\tilde\clif^\kappa\otimes E_2)= {\rm Bott}(\kgot^*)\otimes E_1
$$
where $E_1$ is the module $E_2$ with the action of $K_1$. Then 
\begin{eqnarray*}
\Qcal^\Phi_{K_1\times K_2}(\T^*K)&=&\indice^{\T^*K}_{K_1\times K_2}(\tilde\clif^\kappa)\\
&=&\sum_{\mu\in\what{K}}\indice_{K_1}^{\kgot^*}\left({\rm Bott}(\kgot^*)\otimes V_{\mu}^{K_1}\right)\otimes (V_{\mu}^{K_2})^*\\
&=&\sum_{\mu\in\what{K}} V_{\mu}^{K_1}\otimes (V_{\mu}^{K_2})^*= L^2(K),
\end{eqnarray*}
since $\indice_{K_1}^{\kgot^*}({\rm Bott}(\kgot^*))=1$.

\subsection{The formal quantization of $\T^*(K/H)$.}\label{subsec:Q-K/H}

Let $H$ be a closed connected subgroup of $K$. Look at $T^*K$ as a Hamiltonian manifold relatively to the action of 
$H\times K \subset K_1\times K_2$. The moment map $\Phi=(\Phi_H,\Phi_K)$ is defined by :
$\Phi_H(a,\xi)= -{\rm pr}(\xi)$ and $\Phi_K(a,\xi)= a\cdot\xi$, where ${\rm pr}:\kgot^*\to\hgot^*$ is the projection. 
Note that $\Phi$ is a proper map.

The cotangent bundle $\T^*(K/H)$, viewed as $K$-manifold, is equal to the symplectic reduction of $T^*K$ relatively to the 
$H$-action: if the kernel of the projection ${\rm pr}$ is denoted $\hgot^\perp$, we have
$$
\Phi_H^{-1}(0)/H=K\times_H \hgot^\perp= \T^* (K/H).
$$

We are here in the setting of section \ref{subsec:smooth-K-1-reduction}. The reduction of the 
$H\times K$ proper Hamiltonian manifold $\T^*K$ relatively to the $H$-action is smooth, then its
 formal quantization is computed as follows
\begin{eqnarray}\label{eq:Q-Phi-K-K/H}
\Qcal_{K}^\Phi(\T^*(K/H))=\left[\Qcal_{H\times K}^\Phi(\T^*K)\right]^H
&=&\left[\Qcal_{K_1\times K_2}^\Phi(\T^*K)\vert_{H\times K}\right]^H \nonumber\\
&=& \left[{\rm L}^2(K)\right]^H\\
&=&{\rm L}^2(K/H).\nonumber
\end{eqnarray}
Here the fact that $\Qcal_{H\times K}^\Phi(\T^*K)$ is equal to the restriction of $\Qcal_{K_1\times K_2}^\Phi(\T^*K)=
{\rm L}^2(K)$ to 
$H\times K$ is a consequence of Theorem \ref{theo:pep-formal}.

Let us denoted $\Big[\T^*(K/H)\Big]_{\mu}$ the symplectic reduction at $\mu\in\what{K}$ of 
the K-Hamiltonian 
manifold $\T^*(K/H)$. Theorem \ref{theo:intro} together with (\ref{eq:Q-Phi-G-G/H}) gives
\begin{coro}
For any $\mu\in\what{K}$, we have 
$$
\Qcal\left(\Big[\T^*(K/H)\Big]_{\mu}\right)=\dim \Big[V_\mu^K\Big]^H,
$$
where $[V_\mu^K]^H$ is the subspace of $H$-invariant vector.
\end{coro}

\subsection{The formal quantization of $\T^*(K/H)$ relatively to the action of $G$.}\label{subsec:Q-K/H-G}

Let $G$ be a closed connected subgroup of $K$. We look at the hamiltonian action of G on $\T^*(K/H)$. Let 
$\Phi_G: \T^*(K/H)\to \ggot^*$ be the moment map. We consider also the restriction of the K-module  
${\rm L}^2(K/H)$ to $G$.

We have 

\begin{prop} The following statements are equivalent
\begin{enumerate}
\item The moment map $\Phi_G: \T^*(K/H)\to \ggot^*$ is proper.
\item $\Phi_G^{-1}(0)$ is equal to the zero section.
\item $k\cdot\ggot+\hgot=\kgot$, for any $k\in K$.
\item $\ggot+\hgot=\kgot$
\item $G$ acts transitively on $K/H$.
\item $L^2(K/H)]^G\simeq \C$
\item ${\rm L}^2(K/H)\vert_G$ is an admissible $G$-representation.
\end{enumerate}
\end{prop}

\begin{proof}
$(1)\Longrightarrow (7)$ is a consequence of Theorem \ref{theo:pep-formal}. Let us prove that $(7)\Longrightarrow (6)$. Suppose now that 
$$
{\rm L}^2(K/H)\vert_G=\sum_{\mu\in\what{K}} [V_{\mu}^{K}]^H\otimes (V_{\mu}^{K})^*\vert_G
$$
is an admissible $G$-representation. It means that for any $\lambda\in\what{G}$ the set 
$$
A_\lambda:=\left\{
\mu\in\what{K}\ \vert\ [V_{\mu}^{K}]^H\neq \{0\}\ {\rm and}\  [(V_\lambda^G)^*\otimes (V_{\mu}^{K})^*\vert_G]^G\neq \{0\}
\right\}
$$
is finite. Then the vector space $L^2(K/H)]^G$ is equal to the finite dimensional vector space
$\sum_{\mu\in A_0} [V_{\mu}^{K}]^H\otimes [(V_{\mu}^{K})^*]^G$.  It is not difficult to check that if $\mu\in A_0$, 
then $k\mu\in A_0$ for $k>>1$. Finally the fact that $A_0$ is finite implies that $A_0$ is reduced to $\mu=0$. Hence 
the only $G$-invariant functions on $K/H$ are the scalars.

$(6)\Longleftrightarrow (5)\Longleftrightarrow (4)\Longleftrightarrow (3)$ is a general fact concerning smooth 
actions of a compact connected Lie group $G$ on a compact connected  manifold $M$. The manifold $M$ does not have 
$G$-invariant functions which are not scalar if and only if the action of $G$ on $M$ is transitive. And given a point 
$m\in M$, the orbit $G\cdot m$ is all of $M$ if and only if tangent spaces  $\T_m(G\cdot m)$ and $\T_m M$ are equal.
If we take $m=\overline{k^{-1}}$ in $M=K/H$, the condition $\T_m(G\cdot m)=\T_m M$ is equivalent to $k\cdot\ggot+\hgot=\kgot$.

Let us check $(3)\Longrightarrow (2)$. Let $[k,\xi]\in K\times_H \hgot^\perp= \T^* (K/H)$.We have $\Phi_G([k,\xi])=0$ 
if and only if $k\cdot\xi\in\ggot^\perp$. Hence the vector $\xi$ belongs to 
$$
k^{-1}\cdot\ggot^\perp\bigcap\hgot^\perp=(k^{-1}\cdot\ggot+\hgot)^\perp.
$$
Hence condition $(3)$ imposes that $\xi=0$.

$(2)\Longleftrightarrow (1)$ comes from the fact that $\Phi_G$ is a homogeneous map of degree one between the 
vector bundle $\T^*(K/H)$ and the vector space $\ggot^*$.
\end{proof}

\bigskip

Suppose now that the cotangent bundle $\T^*(K/H)$ is a {\em proper} Hamiltonian $G$-manifold. 
 Let us denoted $[\T^*(K/H)]_{\mu,G}$ the (compact) symplectic reduction at $\mu\in\what{G}$ of 
the G-Hamiltonian manifold $\T^*(K/H)$.    Then,

\begin{coro}
The multiplicity of $V_\mu^G$ in ${\rm L}^2(K/H)$ is 
equal to the quantization of the reduced space $[\T^*(K/H)]_{\mu,G}$.
\end{coro}

\begin{proof}
Using Theorem \ref{theo:pep-formal}, equality (\ref{eq:Q-Phi-K-K/H}) gives then 
\begin{eqnarray*}
\qfor_G(\T^*(K/H))&=&\qfor_K(\T^*(K/H))\vert_G\\
 &=& {\rm L}^2(K/H)\vert_G.
\end{eqnarray*}
In other words, the multiplicity of $V_\mu^G$ in ${\rm L}^2(K/H)$ is equal to the quantization of the 
reduced space $[\T^*(K/H)]_{\mu,G}$.
\end{proof}


\end{document}